\UseRawInputEncoding
\documentclass[11pt]{amsart}
\usepackage{palatino, mathpazo}
\usepackage{amsfonts}
\usepackage{amsmath}
\usepackage{amssymb,latexsym}
\usepackage{graphicx}
\usepackage[mathscr]{eucal}
\usepackage{harpoon}
\usepackage{amssymb}%
\providecommand{\U}[1]{\protect \rule{.1in}{.1in}}

\newtheorem{theorem}{Theorem}[section]

\newtheorem{conjecture}[theorem]{Conjecture}
\newtheorem{corollary}[theorem]{Corollary}

\newtheorem{lemma}[theorem]{Lemma}

\theoremstyle{remark}
\newtheorem{remark}[theorem]{Remark}

\numberwithin{equation}{section}
\begin{document}
\title[$SU(3)$ Toda system on flat tori]{On number and evenness of solutions\\ of the $SU(3)$ Toda system on flat tori\\ with non-critical parameters}
\author{Zhijie Chen}
\address{Department of Mathematical Sciences, Yau Mathematical Sciences Center,
Tsinghua University, Beijing, 100084, China}
\email{zjchen2016@tsinghua.edu.cn}
\author{Chang-Shou Lin}
\address{Center for Advanced Study in
Theoretical Sciences, Taiwan University, Taipei 10617, Taiwan }
\email{cslin@math.ntu.edu.tw}

\begin{abstract}
We study the $SU(3)$ Toda system with singular sources
\[
\begin{cases}
\Delta u+2e^{u}-e^v=4\pi\sum_{k=0}^m n_{1,k}\delta_{p_k}\quad\text{ on
}\; E_{\tau},\\
\Delta v+2e^{v}-e^u=4\pi \sum_{k=0}^m n_{2,k}\delta_{p_k}\quad\text{ on
}\; E_{\tau},
\end{cases}
\]
where $E_{\tau}:=\mathbb{C}/(\mathbb{Z}+\mathbb{Z}\tau)$ with $\operatorname{Im}\tau>0$ is a flat torus, $\delta_{p_k}$ is the Dirac measure at $p_k$, and $n_{i,k}\in\mathbb{Z}_{\geq 0}$ satisfy $\sum_{k}n_{1,k}\not\equiv \sum_k n_{2,k} \mod 3$. This is known as the non-critical case and it follows from a general existence result of \cite{BJMR} that solutions always exist. In this paper we prove that

(i) The system has at most \[\frac{1}{3\times 2^{m+1}}\prod_{k=0}^m(n_{1,k}+1)(n_{2,k}+1)(n_{1,k}+n_{2,k}+2)\in\mathbb{N}\] solutions.
We have several examples to indicate that this upper bound should be sharp. Our proof presents a nice combination of the apriori estimates from analysis and the classical B\'{e}zout theorem from algebraic geometry.

(ii) For $m=0$ and $p_0=0$, the system has even solutions if and only if at least one of $\{n_{1,0}, n_{2,0}\}$ is even. Furthermore, if $n_{1,0}$ is odd, $n_{2,0}$ is even and $n_{1,0}<n_{2,0}$, then except for finitely many $\tau$'s modulo $SL(2,\mathbb{Z})$ action, the system has exactly $\frac{n_{1,0}+1}{2}$ even solutions.

Differently from \cite{BJMR},
our proofs are based on the integrability of the Toda system, and also imply
a general non-existence result for even solutions of the Toda system with four singular sources.
\end{abstract}

\maketitle

\section{Introduction}

Throughout the paper, we use the notations $\omega_{0}=0$, $\omega_{1}=1$,
$\omega_{2}=\tau$, $\omega_{3}=1+\tau$ and $\Lambda_{\tau}=\mathbb{Z+Z}\tau$,
where $\tau \in \mathbb{H}=\{  \tau|\operatorname{Im}\tau>0\}  $.
Define $E_{\tau}:=\mathbb{C}/\Lambda_{\tau}$ to be a flat torus and $E_{\tau}[2]:=\{ \frac{\omega_{k}}{2}|k=0,1,2,3\}+\Lambda
_{\tau}$ to be the set consisting of the lattice points and $2$-torsion points
of $E_{\tau}$.

\subsection{The $SU(3)$ Toda system}
In this paper, we study the following $SU(3)$ Toda system with arbitrary singular sources
\begin{equation}\label{Toda-more-general}
\begin{cases}
\Delta u+2e^{u}-e^v=4\pi\sum_{k=0}^m n_{1,k}\delta_{p_k}\quad\text{ on
}\; E_{\tau},\\
\Delta v+2e^{v}-e^u=4\pi \sum_{k=0}^m n_{2,k}\delta_{p_k}\quad\text{ on
}\; E_{\tau},
\end{cases}
\end{equation}
where $m\geq 0$, $p_0,\cdots,p_m$ are $m+1$ distinct points in $E_{\tau}$, $\delta_{p_k}$ is the Dirac measure at the point $p_k$
and $n_{j,k}\in \mathbb{Z}_{\geq 0}$ for all $j,k$ with $\sum_{j,k}n_{j,k}\geq 1$. We always use the complex variable $z=x_1+ix_2$. Then the Laplace operator $\Delta=4\partial_{z\bar{z}}$.

The $SU(3)$ Toda system (\ref{Toda-more-general}) or its general $SU(N+1)$ version is closely related to the classical infinitesimal Pl\"{u}cker formula in algebraic geometry \cite{GH-AG}; see e.g. \cite{LWYZ,LWY}.
Besides, by denoting
\begin{equation}
\mathcal{N}_j:=\sum_{k=0}^m n_{j,k}\in\mathbb{Z}_{\geq 0},\quad j=1,2,
\end{equation}
 the $SU(3)$ Toda system (\ref{Toda-more-general}) can be also written as
\begin{equation}\label{Toda-mf1}
\begin{cases}
\Delta u+2\rho_1(\frac{e^{u}}{\int e^{u}}-\frac{1}{|E_{\tau}|})-\rho_2(\frac{e^{v}}{\int e^{v}}-\frac{1}{|E_{\tau}|})=4\pi \sum_{k=0}^m n_{1,k}(\delta_{p_k}-\frac{1}{|E_{\tau}|})\text{ on
} E_{\tau},\\
\Delta v+2\rho_2(\frac{e^{v}}{\int e^{v}}-\frac{1}{|E_{\tau}|})-\rho_1(\frac{e^{u}}{\int e^{u}}-\frac{1}{|E_{\tau}|})=4\pi \sum_{k=0}^m n_{2,k}(\delta_{p_k}-\frac{1}{|E_{\tau}|})\text{ on
} E_{\tau},\\
\text{where }\;\rho_1=4\pi \frac{2\mathcal{N}_1+\mathcal{N}_2}{3}>0,\quad \rho_2=4\pi \frac{\mathcal{N}_1+2\mathcal{N}_2}{3}>0,
\end{cases}
\end{equation}
which is a special case of the following general $SU(3)$ Toda system of mean field type
on a compact Riemann surface $\Sigma$:
\begin{equation}
\label{1.1}
\begin{cases}
\Delta u+2\rho_1(\frac{h_1e^{u}}{\int_{\Sigma} h_1 e^{u}}-\frac{1}{|\Sigma|})-\rho_2(\frac{h_2e^{v}}{\int_{\Sigma} h_2e^{v}}-\frac{1}{|\Sigma|})=4\pi \sum_{k=0}^m\alpha_{1,k}(\delta_{p_k}-\frac{1}{|\Sigma|}),\\
\Delta v+2\rho_2(\frac{h_2e^{v}}{\int_{\Sigma} h_2 e^{v}}-\frac{1}{|\Sigma|})-\rho_1(\frac{h_1e^{u}}{\int_{\Sigma} h_1e^{u}}-\frac{1}{|\Sigma|})=4\pi \sum_{k=0}^m\alpha_{2,k}(\delta_{p_k}-\frac{1}{|\Sigma|}).
\end{cases}
\end{equation}
Here $|\Sigma|$ denotes the area of $\Sigma$, $h_1, h_2$ are positive smooth functions on $\Sigma$ and $\alpha_{j,k}>-1$ for all $j,k$. System (\ref{1.1}) and its $SU(N+1)$ version arise in many geometric and physical problems. On the geometric side, the Toda system (\ref{1.1}) has deep relations to holomorphic curves in $\mathbb{CP}^2$, flat $SU(3)$ connection, complete integrability and harmonic sequences. See e.g. \cite{bw,c,cw,g,LNW,LWY} and references therein. While in mathematical physics it arises from the non-abelian Chern-Simons theory which
describes the physics of high critical temperature superconductivity; see e.g. \cite{CL-AIM,CKL2017,CKL2016,CKL2015,LY-2013,nt2,nt3,Yang2001} and references therein. The singularities represent the ramification points of the complex curves and the vortices of the wave functions respectively.

For the Toda system (\ref{1.1}), the existence of solutions is a challenging problem and has been widely studied in the literature; see \cite{B2015,BJMR,BM2016,JLW,LLWY,LN2002,MN2007,MR2013} and references therein. Remark that in these works, the \emph{apriori estimates} are needed to apply either the variational method or Leray-Schauder degree method. Due to this reason, for given singular parameters $\alpha_{j,k}$'s, the parameters $(\rho_1,\rho_2)$ are called \emph{non-critical} (resp. \emph{critical}) if the apriori estimates hold (resp. fail).
For our purpose, let us consider the special case \[\alpha_{j,k}\in \mathbb{Z}_{\geq 0}\quad\text{ for all $j,k$}.\] Then by studying the bubbling phenomena of (\ref{1.1}), it was proved in \cite{LWYZ,LWZ} that $(\rho_1,\rho_2)$ are \emph{non-critical} as long as $\rho_j\not\in 4\pi \mathbb{N}$ for $j=1,2$.

\medskip
\noindent{\bf Theorem A.}  \cite{LWYZ,LWZ} \emph{Let $\alpha_{j,k}\in \mathbb{Z}_{\geq 0}$ for all $j,k$, and let $K\subset \Sigma\setminus \{p_k\}_{k=0}^m$ be any compact set. If $\rho_j\not\in 4\pi \mathbb{N}$ for $j=1,2$, then there is $C=C(K, \rho_1, \rho_2)$ such that for any solution $(u, v)$ of (\ref{1.1}),}
\begin{equation}\label{eq-71}|u(z)|+|v(z)|\leq C,\quad \forall\,z\in K.\end{equation}

By applying the a priori estimate established in Theorem A, it was proved via variational methods in \cite{BJMR} that

\medskip
\noindent{\bf Theorem B.}  \cite{BJMR} \emph{Let $\alpha_{j,k}\in \mathbb{Z}_{\geq 0}$ for all $j,k$. If the genus of $\Sigma$ is positive and $\rho_j\not\in 4\pi \mathbb{N}$ for $j=1,2$, then (\ref{1.1}) has solutions.}

\medskip

In view of Theorems A-B, the first basic question that interests us is:

\medskip
\noindent{\bf Question 1}: {\it Let $n_{j,k}\in \mathbb{R}_{>-1}$ for all $j,k$ and $(\rho_1,\rho_2)$ be non-critical (i.e. the apriori estimates hold). Does the Toda system (\ref{Toda-more-general}) or equivalently (\ref{Toda-mf1})  have only finitely many solutions}?
\medskip

For general elliptic PDEs, it is known that the apriori estimates do not necessarily imply the finiteness of number of solutions (A well-known example is $-\Delta u+u=u^p, u>0$ in $\mathbb{R}^N$ with $1<p<\frac{N+2}{(N-2)_+}$, for which the apriori estimate holds but the dimension of the solution space is $N$ due to the invariance of translation).
We strongly believe that the answer of Question 1 for the Toda system (\ref{Toda-more-general}) is positive. However, even for such Toda system, \emph{the finiteness of the number of solutions is a highly non-trivial question from the analytic point of view}. For a single equation with exponential nonlinearity, there are only a few results to answer Question 1 (see e.g. \cite{CLW}).  In this paper, we initiate to study this question for the Toda system (\ref{Toda-more-general}) with $n_{j,k}\in \mathbb{Z}_{\geq 0}$.

Note for the Toda system (\ref{Toda-more-general}) or equivalently (\ref{Toda-mf1}), we have
\[\rho_1=4\pi \frac{2\mathcal{N}_1+\mathcal{N}_2}{3}>0,\quad \rho_2=4\pi \frac{\mathcal{N}_1+2\mathcal{N}_2}{3}>0,\]
so
\[\rho_1,\rho_2\notin 4\pi\mathbb{N}\quad\text{if and only if }\; \mathcal{N}_1\not\equiv \mathcal{N}_2\mod 3.\]
Thus it follows from Theorem B that

\medskip
\noindent{\bf Theorem C.}  \emph{Let $n_{j,k}\in\mathbb{Z}_{\geq 0}$ for all $j,k$ and $\mathcal{N}_1\not\equiv \mathcal{N}_2\mod 3$, then the Toda system (\ref{Toda-more-general}) always has solutions.}

\medskip

Our first result of this paper reads as follows.

\begin{theorem}\label{thm6} Let $n_{j,k}\in\mathbb{Z}_{\geq 0}$ for all $j,k$ and $\mathcal{N}_1\not\equiv \mathcal{N}_2\mod 3$. Then the Toda system (\ref{Toda-more-general}) has at most
\[N(\{n_{1,k}\}_{k},\{n_{2,k}\}_{k}):=\frac{1}{3\times 2^{m+1}}\prod_{k=0}^m(n_{1,k}+1)(n_{2,k}+1)(n_{1,k}+n_{2,k}+2)\]
solutions.
\end{theorem}

It is elementary to see that $N(\{n_{1,k}\}_{k},\{n_{2,k}\}_{k})\in\mathbb{N}$  because
 \[2|(n_{1,k}+1)(n_{2,k}+1)(n_{1,k}+n_{2,k}+2)\quad\forall k,\]
and $\mathcal{N}_1\not\equiv \mathcal{N}_2\mod 3$ implies $n_{1,k_0}\not\equiv n_{2,k_0}\mod 3$ for some $k_0$ and so exactly one of $\{n_{1,k_0}+1, n_{2,k_0}+1, n_{1,k_0}+n_{2,k_0}+2\}$ is a multiple of $3$, i.e.
\[6|(n_{1,k_0}+1)(n_{2,k_0}+1)(n_{1,k_0}+n_{2,k_0}+2).\]
Theorem \ref{thm6} not only answers Question 1 for the integer case $n_{j,k}\in\mathbb{Z}_{\geq 0}$, but also gives an explicit upper bound for the number of solutions.
We will see from some explicit examples below (see Conjecture \ref{conj-1} and Theorem \ref{thm4}-\ref{thm5}) that this upper bound should be optimal.

\begin{remark}
The topological Leray-Schauder degree is another notion in analysis to describe the "number" of solutions. Since the apriori estimates hold by Theorem A, the topological Leray-Schauder degree for the Toda system (\ref{Toda-more-general}) under the assumption of Theorem \ref{thm6} is well-defined (cf. \cite{LLWY}). We propose
\end{remark}

\begin{conjecture}
Under the same condition as Theorem \ref{thm6}, the topological Leray-Schauder degree for the Toda system (\ref{Toda-more-general}) equals to $N(\{n_{1,k}\}_{k},\{n_{2,k}\}_{k})$.
\end{conjecture}

Next we consider the Toda system (\ref{Toda-more-general}) with singular sources at the lattice points and half periods (i.e. $m=3$ and $p_k=\frac{\omega_k}{2}$):
\begin{equation}\label{Toda-general}
\begin{cases}
\Delta u+2e^{u}-e^v=4\pi\sum_{k=0}^3 n_{1,k}\delta_{\frac{\omega_k}{2}}\quad\text{ on
}\; E_{\tau},\\
\Delta v+2e^{v}-e^u=4\pi \sum_{k=0}^3 n_{2,k}\delta_{\frac{\omega_k}{2}}\quad\text{ on
}\; E_{\tau}.
\end{cases}
\end{equation}
An interesting property of the Toda system (\ref{Toda-general}) is that if $(u(z),v(z))$ is a solution, then so does $(u(-z),v(-z))$. Thus a natural question arises:

\medskip
\noindent{\bf Question 2}: {\it Does the Toda system (\ref{Toda-general}) have even solutions?}
\medskip

At the first sight Question 2 looks simple: One might think that the answer should be positive by restricting the associated functional on the even function subspace of the Sobolev space and applying the variational argument in \cite{BJMR}. However, it turns out that the method in \cite{BJMR} can not work for even solutions in general.

Our second purpose in this paper is to study Question 2.
We will see that Question 2 is not trivial either and its answer could be negative for some cases. Here is our second result.

\begin{theorem}[Nonexistence of even solutions]\label{thm1}
Let $\mathcal{N}_1\not\equiv \mathcal{N}_2\mod 3$. If $n_{1,k},n_{2,k}$ are both odd for some $k\in \{0,1,2,3\}$, then the Toda system (\ref{Toda-general}) has no even solutions.

In other words, the solution obtained in Theorem C can not be even and so the Toda system (\ref{Toda-general}) has at least two solutions.
\end{theorem}

Our next result shows that the assumption of Theorem \ref{thm1} is \emph{sharp} at least for the special case $n_{i,k}=0$ for $k=1,2,3$, i.e. the following system with one singular source
\begin{equation}\label{Toda}
\begin{cases}
\Delta u+2e^{u}-e^v=4\pi n_{1}\delta_{0}\quad\text{ on
}\; E_{\tau},\\
\Delta v+2e^{v}-e^u=4\pi  n_{2}\delta_{0}\quad\text{ on
}\; E_{\tau},
\end{cases}
\end{equation}
where $n_1,n_2\in\mathbb{Z}_{\geq 0}$ satisfy the non-critical condition $n_1\not\equiv n_2\mod 3$. Furthermore, without loss of generality we always assume $n_1< n_2$.

Theorem \ref{thm1} asserts that the system (\ref{Toda}) has no even solutions if $n_1,n_2$ are both odd. Our next result shows that (\ref{Toda}) has even solutions if at least one of $n_1,n_2$ is even.
Note that in this case, there is \emph{a unique even number} among $\{n_1+1, n_2+1, n_1+n_2+2\}$; we denote it by
\begin{equation}\label{su3-fc20}
2N_e:=\text{\it the unique even number among $\{n_1+1, n_2+1, n_1+n_2+2\}$}.
\end{equation}
Note that $N_e=1$ if $n_i=1$ for some $i$.

\begin{theorem}[Existence of even solutions]\label{thm2}
Let $n_1\not\equiv n_2\mod 3$, $n_1< n_2$ and at least one of $n_1,n_2$ be even, with $N_e$ given by (\ref{su3-fc20}). Then the following hold.
\begin{itemize}
\item[(i)]
The Toda system (\ref{Toda}) always has even solutions, and the number of even solutions is at most $N_e$.
\item[(ii)] The even solution is unique if $n_i=1$ for some $i$.

\item[(iii)] If further one of the following holds:
\begin{itemize}
\item[(1)] $n_1$ is odd, i.e. $N_e=\frac{n_1+1}{2}$,
\item[(2)] $n_2$ is odd (i.e. $N_e=\frac{n_2+1}{2}$) and $n_2-n_1\in \{1,5\}$,
\item[(3)] $n_1$ is even and $n_2=n_1+2$, i.e. $N_e=\frac{n_1+n_2+2}{2}=n_1+2$,
\end{itemize}
then except for finitely many $\tau$'s modulo $SL(2,\mathbb{Z})$ action, the number of even solutions for (\ref{Toda}) is exactly $N_e$.
\end{itemize}
\end{theorem}

Here a matrix $A=\bigl(\begin{smallmatrix}a & b\\
c & d\end{smallmatrix}\bigr)
\in SL(2,\mathbb{Z})$ acting on $\tau$ means the Mobius transformation $A\cdot\tau:=\frac{a\tau+b}{c\tau+d}$, and $\tilde{\tau}\equiv \tau$ modulo $SL(2,\mathbb{Z})$ if $\tilde{\tau}=A\cdot\tau$ for some $A\in SL(2,\mathbb{Z})$. It is known that such two tori $E_{\tau}$ and $E_{\tilde{\tau}}$ are conformally equivalent. We believe that the assertion of Theorem \ref{thm2}-(iii) should holds without the additional conditions (1)-(3); see Remark \ref{rmk} in Section 6.

\begin{remark}
It is interesting to compare Theorems \ref{thm1}-\ref{thm2} with those results in our previous work \cite{CL-Toda}, where we studied the Toda system (\ref{Toda}) in the critical case $n_2=n_1+3l$ with $n_1,l\in\mathbb{Z}_{\geq 0}$ and proved that
\begin{itemize}
\item[(i)] When $n_1$ is odd and $n_2$ is even, the Toda system has no even solutions.
\item[(ii)] When $n_1,n_2$ are both odd, the Toda system has at least one family of $2$-parametric even solutions $(u_{\lambda,\mu},v_{\lambda,\mu})$, $\lambda,\mu>0$.
\item[(iii)] When $n_1$ is even, the existence of even solutions depends on the choice of the period $\tau$. Moreover, once the Toda system has an even solution, then it has a $1$-parametric family of even solutions $(u_{\lambda},v_{\lambda})$, $\lambda>0$.
\end{itemize}
Therefore, the structures of even solutions are completely different between the critical case and the non-critical case.
\end{remark}

For the Toda system (\ref{Toda}), we propose the following conjecture.

\begin{conjecture}\label{conj-1} Let $n_1\not\equiv n_2\mod 3$ and $n_1< n_2$. Then the Toda system (\ref{Toda}) has exactly
\[N(n_1,n_2)=\frac{(n_1+1)(n_2+1)(n_1+n_2+2)}{6}\]
solutions except for finitely many $\tau$'s modulo $SL(2,\mathbb{Z})$.
\end{conjecture}

We will give an explanation of Conjecture \ref{conj-1} in Remark \ref{rmkk}. Motivated by Conjecture \ref{conj-1} and Theorem \ref{thm4}-\ref{thm5} below, we propose the following stronger conjecture.

\begin{conjecture}\label{conj-11} Let $n_1\not\equiv n_2\mod 3$ and $n_1< n_2$. Then except for finitely many $\tau$'s modulo $SL(2,\mathbb{Z})$, any solution of the Toda system (\ref{Toda}) is non-degenerate.
\end{conjecture}

It is easy to see that under the assumption of Theorem \ref{thm2}, \[N(n_1,n_2)>N_{e}\quad\text{unless} \quad(n_1,n_2)\in\{(0,1), (0,2)\}.\] This suggests that the Toda system (\ref{Toda}) should have only even solutions for $(n_1,n_2)\in\{(0,1), (0,2)\}$, while for other $(n_1,n_2)$'s even solutions and non-even solutions should exist simultaneously for generic $\tau$'s. Our next results confirm these assertions and Conjecture \ref{conj-1} for special $(n_1,n_2)$'s. Recall that $\wp(z)=\wp(z;\tau)$ is the Weierstrass $\wp$-function with periods $\omega_1=1$ and $\omega_2=\tau$, and $g_2=g_2(\tau), g_3=g_3(\tau)$ are well known invariants of the elliptic curve $E_{\tau}$, given by
\[\wp'(z)^2=4\wp(z)^3-g_2\wp(z)-g_3.\]

\begin{theorem}\label{thm4}\
\begin{itemize}
\item[(1)] If $(n_1,n_2)=(0,1)$, then the solution of the Toda system (\ref{Toda}) is unique and so is even.
\item[(2)] If $(n_1, n_2)=(0,2)$, then for $\tau\equiv e^{\pi i/3}$ modulo $SL(2,\mathbb{Z})$ (i.e. $g_2=0$), the Toda system (\ref{Toda}) has a unique solution which is even; for all other $\tau$'s (i.e. $g_2\neq 0$), the Toda system (\ref{Toda}) has exactly $2$ solutions which are both even.
\end{itemize}

\end{theorem}

\begin{theorem}\label{thm5} Let $(n_1,n_2)=(0,4)$ and note that the weight $12$ modular form $343g_2^3-6561g_3^2$ has a unique zero $\tau_0$ modulo $SL(2,\mathbb{Z})$. Then the following hold.
\begin{itemize}
\item[(1)] For $\tau\equiv i=\sqrt{-1}$ modulo $SL(2,\mathbb{Z})$ (i.e. $g_3=0$), the Toda system (\ref{Toda}) has exactly $3$ solutions that are all even.
\item[(2)] For $\tau\equiv \tau_0$ modulo $SL(2,\mathbb{Z})$, the Toda system (\ref{Toda}) has exactly $4$ solutions, among which $2$ of them are even solutions.
\item[(3)] For $\tau\not\equiv i,\tau_0$ modulo $SL(2,\mathbb{Z})$, the Toda system (\ref{Toda}) has exactly $5$ solutions, among which $3$ of them are even solutions.
\end{itemize}
\end{theorem}

Theorem \ref{thm5} indicates that when $\tau\to \tau_0$, two distinct even solutions converge to the same degenerate even solution; while when $\tau\to i$, three solutions (i.e. two non-even and one even) converge simultaneously to the same degenerate even solution. These statements can be easily seen from the proof of Theorem \ref{thm5} in Section 4.

\subsection{The approach: A third order linear ODE}
Differently from \cite{BJMR},
our approach of proving these results is based on the well-known integrability of the Toda system (\ref{Toda-more-general}).
Let
\[U:=\frac{2u+v}{3},\quad V:=\frac{u+2v}{3},\]
\[\gamma_{1,k}:=\frac{2n_{1,k}+n_{2,k}}{3},\quad \gamma_{2,k}:=\frac{n_{1,k}+2n_{2,k}}{3}.\]
Then (\ref{Toda-more-general}) is equivalent to
\begin{equation}\label{TodaUV}
\begin{cases}
\Delta U+e^{2U-V}=4\pi \sum_{k=0}^m\gamma_{1,k}\delta_{p_k}\quad\text{ on
}\; E_{\tau},\\
\Delta V+e^{2V-U}=4\pi \sum_{k=0}^m\gamma_{2,k}\delta_{p_k}\quad\text{ on
}\; E_{\tau}.
\end{cases}
\end{equation}
Let $(U,V)$ be a solution of (\ref{TodaUV}), then $y(z)=e^{-U(z)}$ solves the following third order linear ODE
\begin{equation}\label{su3-fc1}
(\partial_{z}-V_z)(\partial_{z}+V_z-U_z)(\partial_{z}+U_{z})y(z)=0.
\end{equation}
Recall that \[\zeta(z)=\zeta(z;\tau)=\tfrac{1}{z}-\tfrac{g_2}{60}z^3+\cdots\] is the Weierstrass zeta function defined by $\zeta'(z):=-\wp(z)$. Then we will see in Section 2 that ODE (\ref{su3-fc1}) has the following form
\begin{align}\label{su3-3-ode}
&y'''-\bigg(\sum_{k=0}^m\alpha_k\wp(z-p_k)+\sum_{k=0}^mB_k\zeta(z-p_k)+B\bigg)y'\\
+&\bigg(\sum_{k=0}^m\beta_k\wp'(z-p_k)
+\sum_{k=0}^mD_k\wp(z-p_k)+\sum_{k=0}^mA_k
\zeta(z-p_k)+D\bigg)y=0,\nonumber
\end{align}
where $\alpha_k, \beta_k$ are explicitly determined by $\gamma_{j,k}$'s (see (\ref{su3-fc61})) and $A_k,B_k,D_k,B,D$ are some constants satisfying the residue condition \begin{equation}\label{fc-residue}
\sum_{k=0}^m A_k=0,\quad\sum_{k=0}^m B_k=0.
\end{equation}

This ODE (\ref{su3-3-ode}) has regular singularities at $p_k$ with local exponents
\[-\gamma_{1,k},\quad-\gamma_{1,k}+(n_{1,k}+1),\quad-\gamma_{1,k}+(n_{1,k}+n_{2,k}+2),\]
so (\ref{su3-3-ode}) might have solutions with logarithmic singularity at $p_k$.
Conventionally, the singularity $p_k$ is called \emph{apparent} if (\ref{su3-3-ode}) has no solutions with logarithmic singularity at $p_k$.
It is a standard fact (cf. \cite{LWY,Nie14}) that $p_k$'s are all apparent singularities of (\ref{su3-3-ode}) if it comes from a solution $(U,V)$ of (\ref{TodaUV}); see Section 2 for a brief explanation. Thanks to this fact,
we will prove Theorem \ref{thm1} in Section 5 by showing that at least one of $\frac{\omega_k}{2}$'s can not be apparent of (\ref{su3-3-ode}) under the assumption of Theorem \ref{thm1}.

Denote the set of $3m+5$ parameters of ODE (\ref{su3-3-ode}) by
\begin{equation}\label{vec-B}\vec{\mathbf B}:=(A_0,\cdots,A_m,B_0,\cdots,B_m,B,D_0,\cdots,D_m,D)\end{equation}
for convenience.
Our proof of Theorem \ref{thm6} is based on the following result.

\begin{theorem}\label{thm3-0}
Let $\mathcal{N}_1\not\equiv \mathcal{N}_2\mod 3$. Then there is a one-to-one correspondence between solutions of the Toda system (\ref{Toda-more-general})
and those $\vec{\mathbf B}$'s such that all $p_k$'s are apparent singularities of ODE (\ref{su3-3-ode})-(\ref{fc-residue}).
\end{theorem}

By Theorem \ref{thm3-0}, to prove Theorem \ref{thm6} we only need to count those $\vec{\mathbf B}$'s such that all $p_k$'s are apparent singularities.

For the special case (\ref{Toda}), ODE (\ref{su3-3-ode}) has a simpler expression and reads as
\begin{equation}\label{fc-3ode}
y'''-(\alpha \wp(z)+B)y'+(\beta\wp'(z)+D_0\wp(z)+D)y=0,
\end{equation}
where $B,D_0,D$ are complex constants and
\[
\alpha=\gamma_1(\gamma_1+1)+\gamma_2(\gamma_2+1)-\gamma_1\gamma_2,
\]\[
\beta=-\frac{2\gamma_1(\gamma_1+1)+\gamma_1\gamma_2(\gamma_1-\gamma_2-1)}{2},\]
with
\[\gamma_1:=\frac{2n_1+n_2}{3},\quad \gamma_2:=\frac{n_1+2n_2}{3}.\]
See Section 2 for a proof. Then our proof of Theorems \ref{thm2} and \ref{thm4}-\ref{thm5} is based on the following result.

\begin{theorem}\label{thm3}
Let $n_1\not\equiv n_2\mod 3$ and $n_1 < n_2$. Then there is a one-to-one correspondence between even solutions of the Toda system (\ref{Toda})
and those $(B,D_0,D)$'s with $D_0=D=0$ such that $0$ is an apparent singularity of ODE (\ref{fc-3ode}).
\end{theorem}

Thanks to Theorems \ref{thm3-0}-\ref{thm3}, to prove Theorems \ref{thm2} and \ref{thm4}-\ref{thm5} we only need to count those $(B,D_0,D)$'s (with $D_0=D=0$ for Theorem \ref{thm2}) such that $0$ is an apparent singularity of (\ref{fc-3ode}).

The rest of this paper is organized as follows. In Section 2, we study the precise relation between the Toda system (\ref{TodaUV}) and its associated third order linear ODE (\ref{su3-3-ode}). In Section 3, we prove Theorems \ref{thm3-0}-\ref{thm3}. In Section 4, we prove Theorems \ref{thm6} and \ref{thm4}-\ref{thm5}. Theorems \ref{thm1} and \ref{thm2} will be proved in Sections 5 and 6 respectively. Parts of the above results can be generalized to the general $SU(N+1)$ Toda system and will be studied in a forthcoming paper.

\section{The associated linear ODE}

\label{sec-dual}

In this section, we study the relation between solutions of the Toda system and the associated linear ODE.

Let $(U,V)$ be a solution of the Toda system (\ref{TodaUV}), as in \cite{LWY,Nie14} we
consider the following linear differential operator
\begin{align}\label{operator}
\mathcal{L}:=(\partial_{z}-V_z)(\partial_{z}+V_z-U_z)(\partial_{z}+U_{z})=\partial_{z}^3+W_2\partial_z+W_3,
\end{align}
where
\begin{equation}\label{su3-eqeq}W_2:=U_{zz}-U_z^2+V_{zz}-V_z^2+U_zV_z,\end{equation}
\[W_3:=U_{zzz}-2U_zU_{zz}+U_zV_{zz}+V_zU_z^2-U_zV_z^2\]
are known as $W$-invariants or $W$-symmetries in the literature due to their relationship to the W-algebras; see e.g. \cite{BFO, LWY, Nie12, Nie14, Nie16}. It is easy to see that they satisfy the following crucial property (cf. \cite{LWY,Nie14})
\begin{equation*}
W_{2,\bar{z}}=W_{3,\bar{z}}=0,\quad \text{for }\; z\in E_\tau\setminus \{p_k\}_{k=0}^m,
\end{equation*}
so $W_2$ and $W_3$ are \emph{elliptic} functions. Since (\ref{TodaUV}) yields
\[U(z)=2\gamma_{1,k}\ln|z-p_k|+O(1),\; V(z)=2\gamma_{2,k}\ln|z-p_k|+O(1),\;\text{at }z=p_k,\]
we have
\[W_{2}(z)=
\frac{-\alpha_k}{(z-p_k)^2}+O((z-p_k)^{-1}),\]
\[W_{3}(z)=
\frac{-2\beta_k}{(z-p_k)^3}+O((z-p_k)^{-2}),\]
and $W_2, W_3$ have no other poles in $E_{\tau}$ except $\{p_k\}_{k=0}^m$, where
\begin{equation}\label{su3-fc61}
\alpha_k:=\gamma_{1,k}(\gamma_{1,k}+1)+\gamma_{2,k}(\gamma_{2,k}+1)-\gamma_{1,k}\gamma_{2,k},
\end{equation}
\begin{equation}
\beta_k:=-\frac{2\gamma_{1,k}(\gamma_{1,k}+1)+\gamma_{1,k}\gamma_{2,k}
(\gamma_{1,k}-\gamma_{2,k}-1)}{2}.
\end{equation}
Then the Liouville theorem implies that
\begin{equation}\label{su3-eqeqeq}W_{2}(z)
=-\bigg(\sum_{k=0}^m\alpha_k\wp(z-p_k)+\sum_{k=0}^mB_k
\zeta(z-p_k)+B\bigg),\end{equation}
\[ W_{3}(z)=\sum_{k=0}^m\beta_k\wp'(z-p_k)
+\sum_{k=0}^mD_k\wp(z-p_k)+\sum_{k=0}^mA_k
\zeta(z-p_k)+D,\]
where $ A_k,B_k, D_k, B, D$'s are some constants satisfying the residue condition (\ref{fc-residue}). Indeed, there are $B_k$'s such that
 \[W_{2}(z)+\bigg(\sum_{k=0}^m\alpha_k\wp(z-p_k)+\sum_{k=0}^mB_k
\zeta(z-p_k)\bigg)\]
is holomorphic in $\mathbb{C}$. Since this function is doubly-periodic and so bounded in $\mathbb{C}$, it must be a constant by the Liouville theorem.

In conclusion, it follows from (\ref{operator}) that $e^{-U(z)}$ solves the following third order linear ODE
\begin{align}\label{su3-ode}
&\mathcal{L}y=y'''-\bigg(\sum_{k=0}^m\alpha_k\wp(z-p_k)+\sum_{k=0}^mB_k
\zeta(z-p_k)+B\bigg)y'\\
+&\bigg(\sum_{k=0}^m\beta_k\wp'(z-p_k)
+\sum_{k=0}^mD_k\wp(z-p_k)+\sum_{k=0}^mA_k
\zeta(z-p_k)+D\bigg)y=0.\nonumber
\end{align}

This ODE (\ref{su3-ode}) is of Fuschian type and has regular singularities $p_k$'s.
A direct computation shows that the local exponents $\rho_{k,j}$, $j=1,2,3$, of (\ref{su3-ode}) at $z=p_k$ are
\begin{equation}\label{localex0}
\rho_{k,1}=-\gamma_{1,k}, \quad \rho_{k,2}=-\gamma_{1,k}+(n_{1,k}+1),\end{equation} \[\rho_{k,3}=-\gamma_{1,k}+(n_{1,k}+n_{2,k}+2).\]
Note that under our basic assumption $\mathcal{N}_1\not\equiv\mathcal{N}_2\mod 3$, we see that $\gamma_{1,k}\notin\mathbb{Z}$ for some $k$, so any solution of (\ref{su3-ode}) has branch points at such $p_k$ and hence is \emph{multi-valued}.

On the other hand,
since the exponent differences are all integers, (\ref{su3-ode}) might have solutions with logarithmic singularity at $p_k$'s. As recalled in Section 1, $p_k$ is called {\it an apparent singularity} of (\ref{su3-ode}) if all solutions of (\ref{su3-ode}) has no logarithmic singularity at $p_k$.  In this case, the local monodromy matrix (denoted it by $M_{k}$) at the apparent singularity $p_k$ is given by
\begin{equation}\label{su3-fc2}M_k=e^{-2\pi i \gamma_{1,k}}I_3,\end{equation}
where $I_3=\text{diag}(1,1,1)$ denotes the identity matrix.

\begin{remark}\label{remark11}
Fix a base point $q_{0}=-\varepsilon_0(1+\tau)=-\varepsilon_0\omega_3$ with $0<\varepsilon_0\ll1$ such that $q_0$ is close to $0$ and its neighborhood $B(q_0, |q_0|)=\{z\in\mathbb{C} ||z-q_0|<|q_0|\}$ contains no singularities of (\ref{su3-ode}).
The monodromy representation of (\ref{su3-ode}) is a group homomorphism $\rho:\pi_{1}(  E_{\tau}
\backslash \{p_k\}_{k=0}^m,q_0)  \rightarrow
SL(3,\mathbb{C})$. Let \[\vec{Y}:=(y_1,y_2,y_3)^{T}=\begin{pmatrix}y_1\\
y_2\\
y_3\end{pmatrix}\] be a basis of local solutions in a neighborhood $B(q_0, |q_0|/2)=\{z\in\mathbb{C} ||z-q_0|<|q_0|/2\}$ of $q_0$. For any loop $\ell\in \pi_{1}(  E_{\tau}
\backslash \{p_k\}_{k=0}^m, q_{0}) $, we denote by $\ell^*y_j(z)$ to be the analytic continuation of $y_j(z)$ along $\ell$, then there is a matrix $\rho(\ell)\in SL(3,\mathbb{C})$ such that
\begin{equation}\label{su3-fc30}\ell^*\vec{Y}=\rho(\ell)\vec{Y}.\end{equation}
We say that the monodromy is {\it unitary} if the monodromy group is conjugate to a subgroup of the unitary group $SU(3)$.

Recall $\omega_{1}=1$,
$\omega_{2}=\tau$.
Let $\ell_{j}
\in \pi_{1}(E_{\tau}\backslash \{p_k\}_{k=0}^m,q_{0})$, $j=1,2$, be two
fundamental cycles of $E_{\tau}$ connecting $q_{0}$ with $q_{0}+\omega_{j}$ and let
$\varsigma_{k}\in \pi_{1}(  E_{\tau}%
\backslash \{p_k\}_{k=0}^m  ,q_{0})$ be a
simple loop encircling $p_k$ counterclockwise respectively such that
\begin{equation}
\ell_{1}\ell_{2}\ell_{1}^{-1}\ell_{2}^{-1}=\varsigma_{0}\varsigma_{1}\cdots\varsigma_{m}\text{ \  \ in
}\pi_{1}(  E_{\tau}\backslash \{p_k\}_{k=0}^m
,q_{0})  . \label{II-iv1}%
\end{equation}

Now suppose that $p_k$ are apparent singularities of (\ref{su3-ode}) for all $k$. Then (\ref{su3-fc2}) implies $\rho(\varsigma_{k})=e^{-2\pi i \gamma_{1,k}}I_3$ for all $k$, so we conclude from (\ref{II-iv1}) that the monodromy matrices $N_j:=\rho(\ell_j)$ satisfy
\begin{equation}\label{su3-fc3}N_1N_2N_1^{-1}N_2^{-1}=\varepsilon I_3,\quad\text{where }\varepsilon:=e^{-2\pi i \sum_k \gamma_{1,k}}=e^{-2\pi i \frac{2\mathcal{N}_1+\mathcal{N}_2}{3}}.\end{equation}
Since $\mathcal{N}_1\not\equiv\mathcal{N}_2 \mod 3$, we have $\varepsilon\neq \pm1 $ and $\varepsilon^3=1$.
\end{remark}

\begin{lemma}\label{lemma-su3-1}
Suppose that $p_k$ are apparent singularities of (\ref{su3-ode}) for all $k$. Then up to multiplying by a common constant $c\neq 0$, there is a unique basis of local solutions $\vec{Y}=(y_1,y_2,y_3)^{T}$ in $B(q_0, |q_0|/2)$ such that the associated monodromy matrices $N_1, N_2$ are given by
\begin{equation}\label{su3-fc4}N_1=\left(\begin{smallmatrix}1 & &\\
&\varepsilon & \\
& & \varepsilon^2\end{smallmatrix}\right),\quad N_2=\left(\begin{smallmatrix} & &1\\
1& & \\
&1 & \end{smallmatrix}\right),\end{equation}
where $\varepsilon=e^{-2\pi i \frac{2\mathcal{N}_1+\mathcal{N}_2}{3}}$.
In particular, the monodromy is unitary.

Conversely, if the monodromy is unitary, then $p_k$ are apparent singularities of (\ref{su3-ode}) for all $k$ and so (\ref{su3-fc4}) holds.
\end{lemma}

\begin{proof} Suppose that $p_k$ are apparent singularities of (\ref{su3-ode}) for all $k$. Then we have (\ref{su3-fc3}) holds, i.e. $N_1N_2=\varepsilon N_2N_1$. Let $\lambda$ be an eigenvalue of $N_1$ with eigenvector $x\neq 0$, then $\lambda\neq 0$, $N_2x\neq0$ and
\[N_1N_2x=\varepsilon N_2N_1x=\varepsilon\lambda N_2x,\]
so $\lambda, \varepsilon \lambda, \varepsilon^2\lambda$ are all the eigenvalues of $N_1$. Consequently, $\lambda^3=\lambda^3\varepsilon^3=\det N_1=1$, which implies $\lambda\in \{1,\varepsilon,\varepsilon^2\}$. Thus $1, \varepsilon, \varepsilon^2$ are all the eigenvalues of $N_1$, so up to common conjugation to $N_1, N_2$, we may assume
\[N_1=\left(\begin{smallmatrix}1 & &\\
&\varepsilon & \\
& & \varepsilon^2\end{smallmatrix}\right).\]
Inserting this into $N_1N_2=\varepsilon N_2N_1$, a direct computation leads to
\[N_2=\left(\begin{smallmatrix} & &a\\
b& & \\
&c & \end{smallmatrix}\right)\quad\text{with }abc=1.\]
Letting $P=\text{diag}(1, b^{-1}, a)$, we obtain
\[PN_1P^{-1}=N_1=\left(\begin{smallmatrix}1 & &\\
&\varepsilon & \\
& & \varepsilon^2\end{smallmatrix}\right),\]
\[PN_2P^{-1}=\left(\begin{smallmatrix} 1& &\\
&\frac1b & \\
& & a\end{smallmatrix}\right)\left(\begin{smallmatrix} & &a\\
b& & \\
&c & \end{smallmatrix}\right)\left(\begin{smallmatrix} 1& &\\
&b & \\
& & \frac1a\end{smallmatrix}\right)=\left(\begin{smallmatrix} & &1\\
1& & \\
&1 & \end{smallmatrix}\right).\]
This proves (\ref{su3-fc4}). In particular, $N_1, N_2\in SU(3)$. Since the monodromy group is generated by $N_1, N_2$ and $\rho(\varsigma_{k})=e^{-2\pi i \gamma_{1,k}}I_3\in SU(3)$, $0\leq k\leq m$, we conclude that the monodromy group is contained in $SU(3)$, namely the monodromy is unitary.

If there are two bases of local solutions $(y_1,y_2,y_3)^{T}$ and $(\tilde{y}_1,\tilde{y}_2,\tilde{y}_3)^{T}$ such that the corresponding $N_j$'s are both given by (\ref{su3-fc4}), it follows from $\varepsilon\neq \pm 1$ and the expression of $N_1$ that $\tilde{y}_i=c_iy_i$ with $c_i\neq 0$ for all $i$, and then the expression of $N_2$ implies $c_1=c_2=c_3=:c$, so $(\tilde{y}_1,\tilde{y}_2,\tilde{y}_3)^{T}=c(y_1,y_2,y_3)^{T}$ for some constant $c\neq 0$.

Conversely, if the monodromy is unitary, then $p_k$ are apparent singularities of (\ref{su3-ode}) for all $k$ because of the standard fact (see. e.g. \cite[Section 1.3 in Chapter 1]{GP}): The local monodromy matrix at $p_k$ is of the form
\[e^{-2\pi i \gamma_{1,k}}\left(\begin{smallmatrix}1 & &\\
a&1 & \\
b& c& 1\end{smallmatrix}\right)\quad \text{with }(a,b,c)\neq (0,0,0),\]
and so can not be a unitary matrix if (\ref{su3-ode}) has solutions with logarithmic singularity at $p_k$.
\end{proof}

\begin{lemma} \label{thm-A} Suppose the Toda system (\ref{TodaUV}) has a solution $(U, V)$. Then $p_k$ are apparent singularities of the associated ODE (\ref{su3-ode}) for all $k$ and so Lemma \ref{lemma-su3-1} applies.
In particular, there are a basis of local solutions $\vec{Y}=(y_1, y_2, y_3)^T$ in $B(q_0, |q_0|/2)$ such that
\begin{equation}\label{fffffff}e^{-U(z)}=|y_1(z)|^2+|y_2(z)|^2+|y_3(z)|^2,\end{equation}
and the monodromy matrices $N_1, N_2$ with respect to $(y_1, y_2, y_3)^T$ are given by (\ref{su3-fc4}).
\end{lemma}

\begin{proof}
The proof can be adapted easily from \cite{LWY} and we sketch it here for the reader's convenience.

Let $(f_1,f_2,f_3)^{T}$ be a basis of local solutions in $B(q_0, |q_0|/2)$ for ODE (\ref{su3-ode}) associated to $(U,V)$. Since $e^{-U}>0$ solves (\ref{su3-ode}), we have
\[e^{-U}=\sum_{i,j=1}^{3}m_{i,j}\overline{f_i}f_j,\]
where $(m_{i,j})$ is a Hermitian matrix. By similar arguments as in \cite[Section 5.3]{LWY}, it follows that the Hermitian matrix $(m_{i,j})$ is positive, so we can write $(m_{i,j})=\overline{P}^{T}P$ for some invertible matrix $P\in GL(3,\mathbb{C})$. Define a new basis $(y_1,y_2,y_3)^T:=P(f_1,f_2,f_3)^T$, then (\ref{fffffff}) holds.

Let $\ell\in \pi_{1}(E_{\tau}\backslash \{p_k\}_{k=0}^m,q_{0})$ and $\rho(\ell)$ be the corresponding monodromy matrix with respect to $\vec{Y}=(y_1,y_2,y_3)^T$, i.e. (\ref{su3-fc30}) holds.
Since $e^{-U(z)}$ is single-valued and doubly periodic, namely $\ell^*e^{-U(z)}=e^{-U(z)}$, we see from (\ref{fffffff}) that
\[|y_1(z)|^2+|y_2(z)|^2+|y_3(z)|^2=(\overline{y_1(z)},\overline{y_2(z)},\overline{y_3(z)})\overline{\rho(\ell)}^T\rho(\ell)
\begin{pmatrix}{y_1(z)}\\
{y_2(z)}\\
{y_3(z)}\end{pmatrix}.\]
This implies $\overline{\rho(\ell)}^T\rho(\ell)=I_3$, i.e. $\rho(\ell)\in SU(3)$. So
the monodromy group with respect to $(y_1,y_2,y_3)^{T}$ is contained in $SU(3)$, namely the monodromy is unitary.
Then Lemma \ref{lemma-su3-1} applies and so $p_k$ are apparent singularities for all $k$. By $N_1,N_2\in SU(3)$ and together with $N_1N_2=\varepsilon N_2N_1$ as done in Lemma \ref{lemma-su3-1}, it is elementary to prove the existence of a unitary matrix $P_1$ such that \[P_1N_1P_1^{-1}=\left(\begin{smallmatrix}1 & &\\
&\varepsilon & \\
& & \varepsilon^2\end{smallmatrix}\right),\quad P_1N_2P_1^{-1}=\left(\begin{smallmatrix} & &1\\
1& & \\
&1 & \end{smallmatrix}\right).\]
Then by replacing $(y_1,y_2,y_3)^T$ with $P_1(y_1,y_2,y_3)^T$, we finally obtain (\ref{fffffff}) with the corresponding $N_j$ given by (\ref{su3-fc4}).
\end{proof}

\section{Proofs of Theorems \ref{thm3-0} and \ref{thm3}}

In this section, we prove Theorems \ref{thm3-0} and \ref{thm3}.

\subsection{Proof of Theorem \ref{thm3-0}} First
we recall the notion of \emph{dual equation} for later usage.
Consider the following third order linear ODE
\begin{equation}\label{thirdode}
y'''+p_1(z)y'+p_0(z)y=0.
\end{equation}
For any two functions $y_1, y_2$, we denote by
\begin{equation}\label{two-Wron}W(y_1,y_2):=y_1'y_2-y_1y_2'\end{equation}
to be the Wronskian. Then for any two linearly independent solutions $y_1, y_2$ of (\ref{thirdode}), a direct computation shows that $W(y_1, y_2)$ solves the following equation
\begin{equation}\label{thirdode-dual}
h'''+p_1(z)h'+(p_1'(z)-p_0(z))h=0.
\end{equation}
Equation (\ref{thirdode-dual}) is called the \emph{dual equation} of (\ref{thirdode}) in the literature. By direct computations it is easy to prove that
\begin{itemize}
\item[$(\star)$]
\emph{if $y_1, y_2, y_3$ are linearly independent solutions of (\ref{thirdode}), then $W(y_1,y_2)$, $W(y_2,y_3)$, $W(y_3, y_1)$ are linearly independent solutions of the dual equation (\ref{thirdode-dual}).}
\end{itemize}

The notion of dual equations appears naturally from the Toda system.
As in (\ref{operator}), we consider the following linear differential operator
\begin{align*}
(\partial_{z}-U_z)(\partial_{z}+U_z-V_z)(\partial_{z}+V_{z})=\partial_{z}^3+W_2\partial_z+\tilde{W}_3,
\end{align*}
where
\[\tilde{W}_3:=V_{zzz}-2V_zV_{zz}+V_zU_{zz}+U_zV_z^2-V_zU_z^2=W_2'-W_3.\]
Therefore, $e^{-V(z)}$ solves
{\allowdisplaybreaks
\begin{align}\label{3ode-dual}
h'''-&\bigg(\sum_{k=0}^m\alpha_k\wp(z-p_k)+\sum_{k=0}^mB_k
\zeta(z-p_k)+B\bigg)h'\\
-&\bigg(\sum_{k=0}^m(\alpha_k+\beta_k)\wp'(z-p_k)
\nonumber\\
+&\sum_{k=0}^m(B_k+D_k)\wp(z-p_k)+\sum_{k=0}^mA_k
\zeta(z-p_k)+D\bigg)h=0,\nonumber
\end{align}
}%
which is exactly the dual equation of (\ref{su3-ode}). Remark that
 the local exponents of the dual equation (\ref{3ode-dual}) at $p_k$ are
\begin{equation}\label{eq-s-7}
-\gamma_{2,k},\quad \gamma_{2,k}+(n_{2,k}+1),\quad \gamma_{2,k}+(n_{1,k}+n_{2,k}+2).
\end{equation}

Recall Lemma \ref{thm-A} that if the Toda system (\ref{TodaUV}) has a solution $(U,V)$, then there exist associated parameters $\vec{\mathbf B}$ such that $p_k$ are apparent singularities of (\ref{su3-ode}) with this $\vec{\mathbf B}$ for all $k$. Therefore, to prove Theorem \ref{thm3-0} we only need to prove the following result.

\begin{theorem}[=Theorem \ref{thm3-0}]\label{thm-4.1} Suppose $p_k$ are apparent singularities of (\ref{su3-ode}) for some $\vec{\mathbf B}$. Then the Toda system (\ref{TodaUV}) has a unique solution $(U,V)$ corresponding to this $\vec{\mathbf B}$.

\end{theorem}

\begin{proof} Under our assumption, it follows from Lemma \ref{lemma-su3-1} that (\ref{su3-ode}) has a basis of local solutions $\vec{Y}=(y_1,y_2,y_3)^T$ in $B(q_0,|q_0|/2)$ such that the associated monodromy matrices $N_1, N_2$ are given by (\ref{su3-fc4}).

{\bf Step 1.}
Define the Wronskian matrix of $(y_1,y_2,y_3)$:
\[W:=\begin{pmatrix}y_1 & y_2 & y_3\\y_1' & y_2' & y_3'\\ y_1'' & y_2'' & y_3''\end{pmatrix}.\]
Then $|W|:=|\det W|$ is a positive constant. Define a positive definite matrix
\[R:=|W|^{-\frac{2}{3}}W
\overline{W}^T.\]
Clearly $\det R=1$.
For $1\leq i\leq 3$, we let $R_{i}$ denote the leading principal minor of $R$ of dimension $i$, and define
\begin{align}e^{-U(z)}:=&\frac{1}{4}R_{1}
=\frac{1}{4}|W|^{-\frac{2}{3}}\sum_{j=1}^3|y_j(z)|^2,\;\, z\in B(q_0,|q_0|/2),
\label{n-eq-52}\end{align}
\[ e^{-V(z)}:=\frac{1}{4}R_{2},\quad z\in B(q_0,|q_0|/2).\]
Since $R$ is positive definite, we have $R_{i}>0$ and so $e^{-U}, e^{-V}$ are well-defined in $B(q_0,|q_0|/2)$.
Since $y_j$'s are all holomorphic in $B(q_0,|q_0|/2)$, it is easy to prove (cf. \cite{LWY, Nie12}) that
\begin{equation}\label{n-Rm}
R_{i}(\partial_{z\bar{z}}R_{i})-(\partial_zR_{i})
(\partial_{\bar{z}}R_{i})=R_{{i-1}}R_{i+1},\;\text{for }  i=1,2,
\end{equation}
where $R_{0}:=1$. Note that $R_{3}=\det R= 1$. Letting $i=1$ in (\ref{n-Rm}) leads to
\begin{align}\label{n-V1-eq}
4e^{-V}&=R_{2}=16[e^{-U}
(\partial_{z\bar{z}}e^{-U})
-(\partial_ze^{-U})(\partial_{\bar{z}}e^{-U})]\\
&=-16e^{-2U}\partial_{z\bar{z}}U=-4e^{-2U}\Delta U,\nonumber
\end{align}
and letting $i=2$ in (\ref{n-Rm}) leads to
\begin{align*}
4e^{-U}&=R_{1}=16[e^{-V}
(\partial_{z\bar{z}}e^{-V})
-(\partial_ze^{-V})(\partial_{\bar{z}}e^{-V})]\\
&=-16e^{-2V}\partial_{z\bar{z}}V=-4e^{-2V}\Delta V.
\end{align*}
This proves that $(U,V)$ is a solution of the Toda system (\ref{TodaUV}) in $B(q_0,|q_0|/2)$. Now since the monodromy group with respect to $\vec{Y}=(y_1,y_2,y_3)^T$ is contained in $SU(3)$, it follows from (\ref{n-eq-52}) that for any loop $\ell\in \pi_1(E_{\tau}\setminus\{0\}, q_0)$, the analytic continuation of $e^{-U(z)}$ along $\ell$ is invariant: $\ell^* e^{-U(z)}=e^{-U(z)}$, so after analytic continuation $e^{-U(z)}$ is well-defined on $E_{\tau}$ and then so does $e^{-V(z)}$. Therefore, $(U,V)$ is a solution of the Toda system (\ref{TodaUV}) in $E_{\tau}\setminus\{p_k\}_{k=0}^m$ (because $e^{-U(z)},e^{-V(z)}<+\infty$ in $E_{\tau}\setminus\{p_k\}_{k=0}^m$).

{\bf Step 2.} We prove that $(U,V)$ is a solution of (\ref{TodaUV}) in $E_{\tau}$.
Clearly
this assertion is equivalent to prove that for each $k$,
\begin{equation}\label{ff11}U(z)=2\gamma_{1,k}\ln |z-p_k|+O(1)\quad\text{as} \; z\to p_k,\end{equation}
\begin{equation}\label{ff11-1}V(z)=2\gamma_{2,k}\ln |z-p_k|+O(1)\quad\text{as} \; z\to p_k.\end{equation}

Fix any $k$. By the argument of analytic continuation in the last paragraph of Step 1, there is a small neighborhood $B(p_k,\varepsilon_k)$ of $p_k$ and a basis of local solutions $(\tilde{y}_1,\tilde{y}_2,\tilde{y}_3)$ in $B(p_k,\varepsilon_k)$ (namely this $(\tilde{y}_1,\tilde{y}_2,\tilde{y}_3)$ defined in $B(p_k,\varepsilon_k)$ is an analytic continuation of the original $(y_1,y_2,y_3)$ defined in $B(q_0, |q_0|/2)$) such that
\begin{align}e^{-U}=\frac{1}{4}|W|^{-\frac{2}{3}}(|\tilde{y}_1|^2+|\tilde{y}_2|^2+|\tilde{y}_3|^2)\quad\text{in }\; B(p_k,\varepsilon_k).
\label{nnn-eq-520}\end{align}

Since $(\tilde{y}_1, \tilde{y}_2, \tilde{y}_3)^{T}$ is a basis of local solutions in $B(p_k,\varepsilon_k)$, it follows from (\ref{localex0}) that
\begin{equation}\label{n-2fz0}\tilde{y}_j(z)\backsim (z-p_k)^{-\gamma_{1,k}}\quad \text{as $z\to p_k$},\quad \text{for some $j\in \{1,2,3\}$},\end{equation}
\begin{equation}\label{n-2fz0-1}(z-p_k)^{\gamma_{1,k}}\tilde{y}_j(z)=O(1)\quad \text{as $z\to p_k$},\quad \text{for all $j$}.\end{equation}
From here and (\ref{nnn-eq-520}) we obtain (\ref{ff11}).

To study the asymptotic behavior of $V$ near $0$, we insert (\ref{nnn-eq-520}) into (\ref{n-V1-eq}), which leads to
\begin{align*}
&\frac{1}{4}e^{-V}
=e^{-U}(\partial_{z\bar{z}}e^{-U})
-(\partial_ze^{-U})(\partial_{\bar{z}}e^{-U})\\
=&\frac{1}{16}|W|^{-\frac{4}{3}}[(\sum |\tilde{y}_j|^2)(\sum |\tilde{y}_j'|^2)-(\sum \tilde{y}_j'\overline{\tilde{y}_j})(\sum \tilde{y}_j\overline{\tilde{y}_j'})]\\
=&\frac{1}{16}|W|^{-\frac{4}{3}}[|W(\tilde{y}_1,\tilde{y}_2)|^2+|W(\tilde{y}_2,\tilde{y}_3)|^2
+|W(\tilde{y}_3, \tilde{y}_1)|^2],
\end{align*}
where as pointed out in (\ref{two-Wron}), $W(\tilde{y}_i, \tilde{y}_j):=\tilde{y}_i'\tilde{y}_j-\tilde{y}_j'\tilde{y}_i$, $i\neq j$, is a nontrivial solution of the dual equation (\ref{3ode-dual}).
Since $W(\tilde{y}_i, \tilde{y}_j)$'s form a basis of (\ref{3ode-dual}), it follows from (\ref{eq-s-7}) that
\[
W(\tilde{y}_i, \tilde{y}_j)(z)\sim (z-p_k)^{-\gamma_{2,k}} \;\, \text{as $z\to p_k$},\;\, \text{for some } (i,j),
\]
\[(z-p_k)^{\gamma_{2,k}}W(\tilde{y}_i, \tilde{y}_j)(z)=O(1)\;\, \text{as $z\to p_k$},\;\, \text{for all } (i,j).\]
Thus (\ref{ff11-1}) holds.
This proves that $(U,V)$ is a solution of (\ref{TodaUV}).

{\bf Step 3.} We prove that this $(U,V)$ is the unique solution of (\ref{TodaUV}) corresponding to this $\vec{\mathbf B}$.

Let $(\tilde{U}, \tilde{V})$ be a solution of (\ref{TodaUV}) corresponding to this $\vec{\mathbf B}$. We need to show $(\tilde{U}, \tilde{V})=(U, V)$.
By Lemmas \ref{lemma-su3-1} and \ref{thm-A}, we have
\begin{align*}
e^{-\tilde{U}(z)}
=a(|y_1|^2+|y_2|^2+|y_3|^2)\quad \text{in  }\;B(q_0,|q_0|/2)
\end{align*}
for some constant $a>0$, so we need to prove $a=\frac14|W|^{-\frac{2}{3}}$.

Indeed, we see from (\ref{n-eq-52}) that
\[R_1=|W|^{-\frac{2}{3}}a^{-1} e^{-\tilde{U}}.\]
Inserting this into (\ref{n-Rm}) with $i=1$, we obtain
\begin{align*}
R_{2}&=|W|^{-\frac{4}{3}}a^{-2}[e^{-\tilde{U}}
(\partial_{z\bar{z}}e^{-\tilde{U}})
-(\partial_ze^{-\tilde{U}})(\partial_{\bar{z}}e^{-\tilde{U}})]\\
&=-|W|^{-\frac{4}{3}}a^{-2}e^{-2\tilde{U}}\partial_{z\bar{z}}\tilde{U}
=\frac{1}{4}|W|^{-\frac{4}{3}}a^{-2}e^{-\tilde{V}}.
\end{align*}
Inserting these into (\ref{n-Rm}) with $i=2$, we obtain
\begin{align*}
|W|^{-\frac{2}{3}}a^{-1} e^{-\tilde{U}}&=R_{1}=\frac{|W|^{-\frac{8}{3}}a^{-4}}{16}[e^{-\tilde{V}}
(\partial_{z\bar{z}}e^{-\tilde{V}})
-(\partial_ze^{-\tilde{V}})(\partial_{\bar{z}}e^{-\tilde{V}})]\\
&=-\frac{|W|^{-\frac{8}{3}}a^{-4}}{16}e^{-2\tilde{V}}\partial_{z\bar{z}}\tilde{V}
=\frac{|W|^{-\frac{8}{3}}a^{-4}}{64}e^{-\tilde{U}},
\end{align*}
so $a=\frac14|W|^{-\frac{2}{3}}$, i.e. $e^{-\tilde{U}}=e^{-U}$ and so $(\tilde{U}, \tilde{V})=(U, V)$.
\end{proof}

\subsection{Proof of Theorem \ref{thm3}}

In this section, we consider the Toda system (\ref{Toda}), which is equivalent to
\begin{equation}\label{TodaUV1}
\begin{cases}
\Delta U+e^{2U-V}=4\pi \gamma_{1}\delta_{0}\quad\text{ on
}\; E_{\tau},\\
\Delta V+e^{2V-U}=4\pi \gamma_{2}\delta_{0}\quad\text{ on
}\; E_{\tau},\\
\gamma_1=\frac{2n_1+n_2}{3},\quad \gamma_2=\frac{n_1+2n_2}{3}.
\end{cases}
\end{equation}
Note that for this system, we have $m=0$ and $p_0=0$, so $\mathcal{N}_j=n_j$ and $n_1\not\equiv n_2\mod 3$, so $\varepsilon=e^{-2\pi i \gamma_1}$. Recall that we always assume $n_1<n_2$.

Clearly the associated ODE (\ref{su3-ode}) becomes (note (\ref{fc-residue}))
\begin{equation}\label{3ode}y'''-(\alpha\wp(z)+B)y'+(\beta\wp'(z)+D_0\wp(z)+D)y=0,
\end{equation}
where
\[
\alpha=\gamma_1(\gamma_1+1)+\gamma_2(\gamma_2+1)-\gamma_1\gamma_2,
\]\[
\beta=-\frac{2\gamma_1(\gamma_1+1)+\gamma_1\gamma_2(\gamma_1-\gamma_2-1)}{2}.\]
Furthermore, if $(U,V)$ is an \emph{even} solution of (\ref{TodaUV1}), then $W_3=\beta\wp'(z)+D_0\wp(z)+D$ is odd, so $D_0=D=0$, i.e. the associated ODE (\ref{3ode}) of an even solution becomes
\begin{equation}\label{3ode-even}y'''-(\alpha\wp(z)+B)y'+\beta\wp'(z)y=0.\end{equation}

The (\ref{3ode}) has a regular singularity $0$ with local exponents
$\rho_{j}\notin\mathbb{Z}$:
\begin{equation}\label{localex0-1}
\rho_{1}=-\gamma_{1}, \;\, \rho_{2}=-\gamma_{1}+(n_{1}+1),\;\, \rho_{3}=-\gamma_{1}+(n_{1}+n_{2}+2),
\end{equation}
so any solution of (\ref{3ode}) has a branch point at $0$. Now the monodromy representation of (\ref{3ode}) is a group homomorphism $\rho:\pi_{1}(  E_{\tau}
\backslash \{0\}, q_{0})  \rightarrow
SL(3,\mathbb{C})$ and
\[
\ell_{1}\ell_{2}\ell_{1}^{-1}\ell_{2}^{-1}=\varsigma_{0}\text{ \  \ in
}\pi_{1}(  E_{\tau}\backslash\{0\}
,q_{0})  .
\]

Again if $0$ is an apparent singularity, then the local monodromy matrix $M_{0}$ is given by
\begin{equation}\label{su3-fc2-0}M_0=e^{-2\pi i \gamma_{1}}I_3=\varepsilon I_3,\end{equation}
which implies $\rho(\varsigma_{0})=\varepsilon I_3$ and again
\begin{equation}\label{su3-fc3-0}N_1N_2N_1^{-1}N_2^{-1}=\varepsilon I_3,\quad\text{i.e. }\;N_1N_2=\varepsilon N_2N_1.\end{equation}

\begin{proof}[Proof of Theorem \ref{thm3}]
By Theorem \ref{thm3-0} and the above argument, we only need to prove that the solution $(U,V)$ constructed in Theorem \ref{thm-4.1} is even if $D_0=D=0$.

Let us turn back to the proof of Theorem \ref{thm-4.1}. Since $\mathbb{Z}+\mathbb{Z}\tau$ is the set of all branch points of ODE (\ref{3ode}), we consider
\[L:=[-1,1]\cup\bigcup_{i=0}^{+\infty}([2+3i,4+3i]\cup[-4-3i,-2-3i])\subset\mathbb{R},\]
and
\[\Xi:=\mathbb{C}\setminus(L+\mathbb{Z}\tau).\]
Then $\Xi$ has no intersection with $\mathbb{Z}+\mathbb{Z}\tau$, and is path-connected, symmetric with respect to $z\leftrightarrow-z$ and $z\leftrightarrow z+\tau$. Note that each line in $\Xi$ contains exactly $3$ points of $\mathbb{Z}+\mathbb{Z}\tau$. By (\ref{su3-fc2-0}) (i.e. the local monodromy matrix at each point of $\mathbb{Z}+\mathbb{Z}\tau$ is $\varepsilon I_3$) and $\varepsilon^3=1$, it follows that the local solutions $y_j(z)$ in $B(q_0,|q_0|/2)$ can be extended to be \emph{single-valued holomorphic} functions (still denoted by $y_j(z)$) in  $\Xi$. Then
\begin{align}e^{-U}=\frac{1}{4}|W|^{-\frac{2}{3}}(|y_1|^2+|y_2|^2+|y_3|^2)\quad\text{in }\; \Xi.
\label{n-eq-520}\end{align}

Since $D_0=D=0$, ODE (\ref{3ode}) becomes (\ref{3ode-even}), i.e. is invariant with respect to $z\leftrightarrow-z$.
Since $\Xi=\mathbb{C}\setminus(L+\mathbb{Z}\tau)$ is symmetric with respect to $z\leftrightarrow-z$, $y_j(-z)$ is also a well-defined solution of (\ref{3ode}) in $\Xi$. So there is a invertible matrix $M$ such that
\begin{equation}\label{su3-fc16}
\vec{Y}(-z)=M\vec{Y}(z),\quad \forall z\in\Xi.
\end{equation}
It is easy to see $\det M=-1$.
It suffices to prove $M\in U(3)$, i.e. $\overline{M}^TM=I_3$. Once this holds, it follows immediately from (\ref{n-eq-520}) that $U(z)=U(-z)$ and so $V(z)=V(-z)$.

Since $\Xi$ is symmetric with respect to $z\leftrightarrow z+\tau$, $y_j(z+\tau)$ is also a well-defined solution of (\ref{3ode}) in $\Xi$. So there is a invertible matrix $\tilde{N}_2$ such that
\begin{equation}
\vec{Y}(z+\tau)=\tilde{N}_2\vec{Y}(z),\quad \forall z\in\Xi.
\end{equation}
By the definition of $\Xi$, the choice of $q_0$ and the fundamental circle $\ell_2$ in Remark \ref{remark11}, we have
$\tilde{N}_2=\varepsilon^2 N_2$, i.e.
\begin{equation}\label{su3-fc12}
\vec{Y}(z+\tau)=\varepsilon^2N_2\vec{Y}(z),\quad \forall z\in\Xi.
\end{equation}
Define
\[\Omega_{+}:=\{a+b\tau \,|\, b\in (0,1), a\in\mathbb{R}\}\subset \Xi,\]
\[\Omega_{-}:=\{a+b\tau \,|\, b\in (-1,0), a\in\mathbb{R}\}\subset \Xi,\]
Since $\Omega_{\pm}=\Omega_{\pm}+1$, $y_j(z+1)$ is also a well-defined solution for $z\in\Omega_{\pm}$, so there are invertible matrices $N_{\pm}$ such that
\[\vec{Y}(z+1)=N_{\pm}\vec{Y}(z),\quad \forall z\in\Omega_{\pm}.\]
By $q_0\in \Omega_-$ and the definition of the fundamental circle $\ell_1$, we have $N_-=N_1$, i.e.
\begin{equation}\label{su3-fc14}\vec{Y}(z+1)=N_{1}\vec{Y}(z),\quad \forall z\in\Omega_{-}.\end{equation}
Recalling the definition of $\Omega$ and (\ref{su3-fc2-0}), it is easy to see that $N_{+}^{-1}N_1=M_0=\varepsilon I_3$, so $N_+=\varepsilon^{-1}N_1$, i.e.
\begin{equation}\label{su3-fc15}\vec{Y}(z+1)=\varepsilon^{-1} N_{1}\vec{Y}(z),\quad \forall z\in\Omega_{+}.\end{equation}
Remark that (\ref{su3-fc15}) can be also proved in a different way:
Since $z-\tau\in\Omega_-$ for $z\in \Omega_+$, we have for $z\in \Omega_+$ that
\begin{align*}
\vec{Y}(z+1)&=\vec{Y}(z-\tau+1+\tau)=\varepsilon^2N_2\vec{Y}(z-\tau+1)\quad \text{by (\ref{su3-fc12})}\\
&=\varepsilon^2 N_2N_1\vec{Y}(z-\tau)\quad \text{by (\ref{su3-fc14})}\\
&=N_2 N_1N_2^{-1}\vec{Y}(z)\quad \text{by (\ref{su3-fc12})}\\
&=\varepsilon^{-1} N_1\vec{Y}(z)\quad \text{by (\ref{su3-fc3-0})}.
\end{align*}

Therefore, for $z\in \Omega_-$ we have
\begin{align*}
\vec{Y}(-(z+1))&=M\vec{Y}(z+1)\quad \text{by (\ref{su3-fc16})}\\
&=M N_1\vec{Y}(z)\quad \text{by (\ref{su3-fc14})},
\end{align*}
and (note $-z\in\Omega_+$)
\begin{align*}
\vec{Y}(-(z+1))&=\vec{Y}(-z-1)
=\varepsilon N_{1}^{-1}\vec{Y}(-z)\quad \text{by (\ref{su3-fc15})}\\
&=\varepsilon N_{1}^{-1}M\vec{Y}(z)\quad \text{by (\ref{su3-fc16})},
\end{align*}
so $M N_1=\varepsilon N_{1}^{-1}M$. From here and (\ref{su3-fc4}), a direct computation gives
\[M=\left(\begin{smallmatrix} &a&\\
b&& \\
& &c \end{smallmatrix}\right)\quad\text{with }abc=-\det M=1.\]
On the other hand, by applying (\ref{su3-fc16}) and (\ref{su3-fc12}) to $\vec{Y}(-(z+\tau))$, we obtain \[M\varepsilon^2 N_2=\varepsilon^{-2} N_2^{-1} M,\] from which and (\ref{su3-fc4}) we easily obtain $c=\varepsilon b =\varepsilon ^2 a$, so $1=abc=\varepsilon^3 a^3=a^3$, i.e. $a\in \{1,\varepsilon,\varepsilon^2\}$ and
\[M=a\left(\begin{smallmatrix} &1&\\
\varepsilon& &\\
& & \varepsilon^2\end{smallmatrix}\right).\]
This proves $M\in U(3)$ and so $(U(z),V(z))$ is an even solution.

The proof is complete.
\end{proof}

\section{Proofs of Theorems \ref{thm6} and \ref{thm4}-\ref{thm5}}

In this section, we prove Theorems \ref{thm6}, \ref{thm4}-\ref{thm5} and also give an explanation of Conjecture \ref{conj-1}.

\subsection{Proof of Theorem \ref{thm6}}
By Theorem \ref{thm3-0} we need to study the condition on \[\vec{\mathbf B}=(A_0,\cdots,A_m,B_0,\cdots,B_m,B,D_0,\cdots,D_m,D)\] such that $p_k$ are apparent singularities of (\ref{su3-3-ode}) for all $k$. Here we apply the standard Frobenius' method (see e.g. \cite[Section 1.3 in Chapter 1]{GP}) to study the apparent condition. In the sequel we use the notations
\[\zeta_{kl}=\zeta(p_k-p_l),\quad \wp_{kl}^{(n)}=\wp^{(n)}(p_k-p_l),\quad\forall n\geq 0,\]
whenever $k\neq l$. We also use the notation $\mathbb{Q}[g_2,g_3,\zeta_{kl},\wp^{(n)}_{kl}][\vec{\mathbf B}]$ to denote
\[\mathbb{Q}[g_2,g_3,\{\zeta_{kl}: l\neq k\},\{\wp^{(n)}_{kl}: l\neq k, n\geq 0\}][\vec{\mathbf B}]\]
just for convenience.

\begin{lemma}\label{lemma-su3-33} Fix any $k$. Then there are three polynomials
\[P_{k,1}(\vec{\mathbf B})=r_1D_k^{n_{1,k}+1}+\cdots\in\mathbb{Q}[g_2,g_3,\zeta_{kl},\wp^{(n)}_{kl}][\vec{\mathbf B}],\;r_1\in\mathbb{Q}\setminus\{0\},\] \[P_{k,2}(\vec{\mathbf B})=r_2D_k^{n_{2,k}+1}+\cdots\in\mathbb{Q}[g_2,g_3,\zeta_{kl},\wp^{(n)}_{kl}][\vec{\mathbf B}],\;r_2\in\mathbb{Q}\setminus\{0\},\]
\[P_{k,3}(\vec{\mathbf B})=r_3A_kD_k^{n_{1,k}+n_{2,k}}+\cdots\in\mathbb{Q}[g_2,g_3,\zeta_{kl},\wp^{(n)}_{kl}][\vec{\mathbf B}],\;r_3\in\mathbb{Q}\setminus\{0\},\]
with homogeneous weights
\begin{equation}
\operatorname{Weig} P_{k,1}=n_{1,k}+1,\; \operatorname{Weig} P_{k,2}=n_{2,k}+1,\; \operatorname{Weig} P_{k,3}=n_{1,k}+n_{2,k}+2,
\end{equation}
such that $p_k$ is an apparent singularity of ODE (\ref{su3-3-ode}) if and only if
 \[P_{k,1}(\vec{\mathbf B})=P_{k,2}(\vec{\mathbf B})=P_{k,3}(\vec{\mathbf B})=0.\]
Here the weights of $B_i,D_i,\zeta_{kl},A_j, B, D, g_2, g_3, \wp_{kl}^{(n)}$ are $1,1,1,2,2,3, 4, 6, n+2$ respectively.
\end{lemma}

\begin{proof}Recalling the local exponents $\rho_{k,j}$ in (\ref{localex0}), it is standard by Frobenius theory that $p_k$ is an apparent singularity if and only if ODE (\ref{su3-3-ode}) has local solutions of the form
\begin{equation}\label{su3-3-fc41}
y(z)=u^{\rho_{k,1}}\sum_{j=0}^{+\infty}c_j u^j,\quad u:=z-p_k
\end{equation}
with $c_0, c_{n_{1,k}+1}, c_{n_{1,k}+n_{2,k}+2}$ being arbitrary (this corresponds to the dimension of solutions being $3$).

Recall the well-known Laurent expansions of $\wp(z),\zeta(z)$:
\[\wp(z-p_k)=\sum_{j=0}^{+\infty}b_j u^{j-2},\quad \zeta(z-p_k)=-\sum_{j=0}^{+\infty}\frac{b_j}{j-1} u^{j-1},\]
where $b_0=1, b_2=0$, $b_{j}=0$ for all odd $j$ (here $\frac{b_j}{j-1}:=0$ for $j=1$) and $b_{j}\in \mathbb{Q}[g_2,g_3]$ is of homogeneous weight $j$ for all even $j\geq 4$. For example, $b_4=\frac{g_2}{20}$, $b_6=\frac{g_3}{28}$ etc.
Furthermore, for $l\neq k$,
\[\zeta(z-p_l)=\zeta_{kl}-\sum_{j=1}^{\infty}\frac{1}{j!}\wp_{kl}^{(j-1)}u^j,\]
\[\wp(z-p_l)=\wp_{kl}^{(0)}+\sum_{j=1}^{\infty}\frac{1}{j!}\wp_{kl}^{(j)}u^j,\]
\[\wp'(z-p_l)=\wp_{kl}^{(1)}+\sum_{j=1}^{\infty}\frac{1}{j!}\wp_{kl}^{(j+1)}u^j.\]
Inserting (\ref{su3-3-fc41}) and these Laurent or Taylor expansions into ODE (\ref{su3-3-ode}), a slightly complicate but direct computation leads to
{\allowdisplaybreaks
\begin{align}\label{eq-appapp}
&\sum_{j=0}^{\infty}\bigg(\phi_jc_j+[D_k-(j+\rho_{k,1}-1)B_k]c_{j-1}\\
&-\Big[(j+\rho_{k,1}-2)B-A_k+(j+\rho_{k,1}-2)\sum_{l\neq k}(\alpha_l\wp_{kl}^{(0)}+B_l\zeta_{kl})\Big]c_{j-2}\nonumber\\
&+\Big[D-(j+\rho_{k,1}-3)\sum_{l\neq k}(\alpha_l\wp_{kl}^{(1)}-B_l\wp_{kl}^{(0)})+\sum_{l\neq k}(\beta_l\wp_{kl}^{(1)}+D_l\wp_{k,l}^{(0)}+A_l\zeta_{kl})\Big]c_{j-3}\nonumber\\
&-\sum_{i=4}^j[(j+\rho_{k,1}-i)\alpha-(i-2)\beta]b_ic_{j-i}
+\sum_{i=4}^{j-1}(D_k+\tfrac{j+\rho_{k,1}-i-1}{i-1}B_k)b_ic_{j-1-i}\nonumber\\
&-\sum_{i=4}^{j-2}\tfrac{A_k}{i-1}b_ic_{j-2-i}
+\sum_{i=1}^{j-3}\tfrac1{i!}\sum_{l\neq k}\big(\beta_l\wp_{kl}^{(i+1)}+D_l\wp_{k,l}^{(i)}-A_l\wp_{kl}^{(i-1)}\big)c_{j-3-i}\nonumber\\
&-\sum_{i=2}^{j-2}(j+\rho_{k,1}-i-2)\sum_{l\neq k}(\alpha_l\wp_{kl}^{(i)}-B_l\wp_{kl}^{(i-1)})c_{j-2-i}\bigg)
u^{j+\rho_{k,1}-3}=0,\nonumber
\end{align}
}%
where $c_{-j}:=0$ for $j\geq 1$ and
\begin{align*}
\phi_j:=&(j+\rho_{k,1})(j+\rho_{k,1}-1)(j+\rho_{k,1}-2)-\alpha_k(j+\rho_{k,1})-2\beta_k\\
=&j(j-n_{1,k}-1)(j-n_{1,k}-n_{2,k}-2).
\end{align*}
Remark that in (\ref{eq-appapp}), the knowledge of the explicit complicate-looking expression of the coefficient of $c_{j-i}$ with $i\geq 3$ is not necessary in our following argument; the only important thing is that \emph{the coefficient of $c_{j-i}$ is of homogeneous weight $i$} under our setting of the weights for $B_l, D_l$ etc.

Therefore, $y(z)$ is a solution of (\ref{su3-3-ode}) if and only if
\begin{align}\label{rec-3-app}
\phi_j c_j=&-[D_k-(j+\rho_{k,1}-1)B_k]c_{j-1}+\Big[(j+\rho_{k,1}-2)B-A_k\\
&+(j+\rho_{k,1}-2)\sum_{l\neq k}(\alpha_l\wp_{kl}^{(0)}+B_l\zeta_{kl})\Big]c_{j-2}+\sum_{i\geq 3}R_{i}c_{j-i},\;\forall j\geq 0,\nonumber
\end{align}
where $R_i\in \mathbb{Q}[g_2,g_3,\zeta_{kl},\wp^{(n)}_{kl}][\vec{\mathbf B}]$ is of homogeneous weight $i$ and its expression can be determined by (\ref{eq-appapp}).

Note that $\phi_j=0$ if and only if $j\in\{0, n_{1,k}+1, n_{1,k}+n_{2,k}+2\}$, and (\ref{rec-3-app}) with $j=0$ holds automatically.
By (\ref{rec-3-app}) and the induction argument, for $1\leq j\leq n_{1,k}$, $c_j$ can be uniquely solved as \[\frac{c_j}{c_0}=c_j(\vec{\mathbf B})=\tilde{r}_j D_k^j+\cdots\in \mathbb{Q}[g_2,g_3,\zeta_{kl},\wp^{(n)}_{kl}][\vec{\mathbf B}]\] with degree $j$ in $\vec{\mathbf B}$ and homogenous weight $j$, where $\tilde{r}_j\in\mathbb{Q}\setminus\{0\}$.

Consequently, there is a polynomial $P_{k,1}(\vec{\mathbf B})$ as stated in this lemma such that the RHS of (\ref{rec-3-app}) with $j=n_{1,k}+1$ equals $c_0P_{k,1}(\vec{\mathbf B})$.
Since $c_0$ can be arbitrary, we obtain $P_{k,1}(\vec{\mathbf B})=0$.

Again by (\ref{rec-3-app}) and the induction argument, for $n_{1,k}+2\leq j\leq n_{1,k}+n_{2,k}+1$, $c_j$ can be uniquely solved as
\[c_j=c_{n_{1,k}+1}c_{j,1}(\vec{\mathbf B})+c_{0}c_{j,2}(\vec{\mathbf B}),\]
where \[c_{j,1}(\vec{\mathbf B})=r_{j,1}D_k^{j-n_{1,k}-1}+\cdots\in \mathbb{Q}[g_2,g_3,\zeta_{kl},\wp^{(n)}_{kl}][\vec{\mathbf B}]\]
is of homogenous weight $j-n_{1,k}-1$, and
\[c_{j,2}(\vec{\mathbf B})=r_{j,2}A_kD_k^{j-2}+\cdots\in \mathbb{Q}[g_2,g_3,\zeta_{kl},\wp^{(n)}_{kl}][\vec{\mathbf B}]\]
is of homogenous weight $j$, and $r_{j,1},r_{j,2}\in\mathbb{Q}\setminus\{0\}$.

Inserting these into (\ref{rec-3-app}) with $j=n_{1,k}+n_{2,k}+2$,
there are two polynomials $P_{k,s}(\vec{\mathbf B})$, $s=2,3$ as stated in this lemma such that
the RHS of (\ref{rec-3-app}) with $j=n_{1,k}+n_{2,k}+2$ equals \[c_{n_{1,k}+1}P_{k,2}(\vec{\mathbf B})+c_{0}P_{k,3}(\vec{\mathbf B}).\]
Since $c_0$ and $c_{n_{1,k}+1}$ can be arbitrary, we obtain $P_{k,s}(\vec{\mathbf B})=0$ for $s=2,3$.

Conversely, if $P_{k,s}(\vec{\mathbf B})=0$ for $s=1,2,3$, then the standard Frobenius theory shows that  (\ref{su3-3-ode}) has local solutions of the form (\ref{su3-3-fc41}) with $c_0, c_{n_{1,k}+1}$, $c_{n_{1,k}+n_{2,k}+2}$ being arbitrary, so $p_k$ is apparent. This completes the proof.
\end{proof}

As a consequence of Theorem \ref{thm3-0} and Lemma \ref{lemma-su3-33}, we obtain

\begin{corollary}\label{Coro-1}
$p_k$ are apparent singularities of ODE (\ref{su3-3-ode}) with (\ref{fc-residue}) for all $0\leq k\leq m$ if and only if $\vec{\mathbf B}\in \mathbb{C}^{3m+5}$ is a solution of the following $3m+5$ polynomials system
\begin{equation}\label{fc-fc11}
\begin{cases}
P_1(\vec{\mathbf B}):=\sum_{k=0}^m B_k=0,\\
P_2(\vec{\mathbf B}):=\sum_{k=0}^m A_k=0,\\
P_{k,1}(\vec{\mathbf B})=P_{k,2}(\vec{\mathbf B})=P_{k,3}(\vec{\mathbf B})=0,\quad 0\leq k\leq m.
\end{cases}
\end{equation}

Consequently,
the number of solutions of the Toda system (\ref{Toda-more-general})
equals to the number of solutions $\vec{\mathbf B}$'s of the polynomial system (\ref{fc-fc11}).
\end{corollary}

Now we are in the position to prove Theorem \ref{thm6}.

\begin{proof}[Proof of Theorem \ref{thm6}]
The key point is to show the existence of a constant $C>1$ such that for any solution $\vec{\mathbf B}$ of (\ref{fc-fc11}), there holds
\begin{equation}\label{su3-fc46}\sum_{k=0}^m|B_k|+\sum_{k=0}^m|A_k|+\sum_{k=0}^m|D_k|+|B|+|D|\leq C.\end{equation}

This assertion seems difficult to prove from the viewpoint of the polynomials. Here we use its connection with the Toda system.
By Theorem \ref{thm-4.1}, we know that for any solution $\vec{\mathbf B}$ of (\ref{fc-fc11}), there is a solution $(U, V)$ of the Toda system (\ref{TodaUV}) corresponding to this $\vec{\mathbf B}$ via (\ref{su3-eqeq}) and (\ref{su3-eqeqeq}).

Let $z_0\in E_{\tau}\setminus\{p_k\}_{k=0}^m$ and fix small $\epsilon>0$ such that the neighborhood $B(z_0, \epsilon)\Subset E_{\tau}\setminus\{p_k\}_{k=0}^m$. Then by Theorem A, there is a constant $C_1$ such that for any solution $(U,V)$ of the Toda system (\ref{TodaUV}), we have
\[|U(z)|+|V(z)|\leq C_1,\quad\forall z\in B(z_0, \epsilon).\]
Then by applying standard gradient estimates to the system (\ref{TodaUV}), there is a constant $C_2$ such that for any solution $(U,V)$ of (\ref{TodaUV}), we have
\[|U_z|+|V_z|+|U_{zz}|+|V_{zz}|+|U_{zzz}|\leq C_2,\quad\forall z\in B(z_0, \epsilon/2).\]
From here, (\ref{su3-eqeq}) and (\ref{su3-eqeqeq}), we see that there is a constant $C_3$ such that for any solution $\vec{\mathbf B}$ of (\ref{fc-fc11}),
\begin{equation}\label{fc-fc12}\left|\sum_{k=0}^mB_k
\zeta(z-p_k)+B\right|\leq C_3,\quad\forall z\in B(z_0, \epsilon/2),\end{equation}
\[\left|\sum_{k=0}^mD_k\wp(z-p_k)+\sum_{k=0}^mA_k
\zeta(z-p_k)+D\right|\leq C_3,\quad\forall z\in B(z_0, \epsilon/2).\]
Here we used the fact that both $\sum_{k=0}^m \alpha_k\wp(z-p_k)$ and $\sum_{k=0}^m \beta_k\wp'(z-p_k)$ are uniformly bounded for $z\in B(z_0, \epsilon/2)$ because $B(z_0, \epsilon)\Subset E_{\tau}\setminus\{p_k\}_{k=0}^m$.

By taking $m+2$ distinct points $z_j\in B(z_0, \epsilon/2)$, $1\leq j\leq m+2$, such that
the $(m+2)\times(m+2)$ matrix
\[\begin{pmatrix}\zeta(z_1-p_0)&\cdots&\zeta(z_1-p_m)&1\\
\vdots&\vdots&\vdots&\vdots\\
\zeta(z_{m+2}-p_0)&\cdots&\zeta(z_{m+2}-p_m)&1
\end{pmatrix}\]
is invertible, we conclude from (\ref{fc-fc12}) that $\sum_{k=0}^m |B_k|+|B|$ is uniformly bounded.
The similar argument also yields the uniform estimates for $\sum_{k=0}^m|A_k|+\sum_{k=0}^m|D_k|+|D|$. This proves (\ref{su3-fc46}).

Now we let \begin{equation}\label{fc-fc13} B=\tilde{B}^2,\quad D=\tilde{D}^3,\quad A_k=\tilde{A}_k^2,\quad\forall 0\leq k\leq m,\end{equation}
and define
\[\vec{\mathcal B}:=(\tilde{A}_0,\cdots,\tilde{A}_m,B_0,\cdots,B_m,\tilde{B},D_0,\cdots,D_m,\tilde{D}),\]
\[\tilde{P}_j(\vec{\mathcal{B}}):=P_j(\vec{\mathbf B}),\quad j=1,2,\]
\[\tilde{P}_{k,j}(\vec{\mathcal{B}}):=P_{k,j}(\vec{\mathbf B}),\quad 0\leq k\leq m,\; j=1,2,3.\]
Then it follows from Lemma \ref{lemma-su3-33} and Corollary \ref{Coro-1} that as polynomials of $\vec{\mathcal{B}}$, there hold
\[\deg \tilde{P}_{1}=1,\quad \deg \tilde{P}_2=2,\]
\[\deg \tilde{P}_{k,1}=n_{1,k}+1,\; \deg \tilde{P}_{k,2}=n_{2,k}+1,\; \deg \tilde{P}_{k,3}=n_{1,k}+n_{2,k}+2.\]
Furthermore, (\ref{su3-fc46}) implies that
\begin{equation}\label{su3-fc46-1}\sum_{k=0}^m|B_k|+\sum_{k=0}^m|\tilde{A}_k|+\sum_{k=0}^m|D_k|+
|\tilde{B}|+|\tilde{D}|\leq C\end{equation}
holds for any solution $\vec{\mathcal{B}}$ of the new polynomial system
\begin{equation}\label{fc-fc111}
\begin{cases}
\tilde{P}_1(\vec{\mathcal B})=\sum_{k=0}^m B_k=0,\\
\tilde{P}_2(\vec{\mathcal B}):=\sum_{k=0}^m \tilde{A}_k^2=0,\\
\tilde{P}_{k,1}(\vec{\mathcal B})=\tilde{P}_{k,2}(\vec{\mathcal B})=\tilde{P}_{k,3}(\vec{\mathcal B})=0,\quad 0\leq k\leq m.
\end{cases}
\end{equation}

Denote $X:=\{ \vec{\mathcal B}\in \mathbb{C}^{3m+5} | \text{$\vec{\mathcal{B}}$ is a solution of (\ref{fc-fc111})}\}$, then $X$ is an affine variety. Since an affine variety over $\mathbb{C}$ can not be bounded in the standard topology except it is a finite set (see e.g. \cite[p.35]{AG}), we conclude from (\ref{su3-fc46-1}) that $X$ consists of finite points, namely
the polynomial system (\ref{fc-fc111}) has only finitely many solutions.
Then it follows from the B\'{e}zout theorem in algebraic geometry (see e.g. \cite[p.246]{IRS}) that the polynomial system (\ref{fc-fc111}) has at most \begin{align*}&\deg \tilde{P}_1\times\deg \tilde{P}_2\times\prod_{k=0}^m\prod_{j=1}^3\deg \tilde{P}_{k,j}\\
=&2\prod_{k=0}^m(n_{1,k}+1)(n_{2,k}+1)(n_{1,k}+n_{2,k}+2)\\
=&3\times 2^{m+2}N(\{n_{1,k}\}_{k},\{n_{2,k}\}_{k})
\end{align*}
solutions by counting multiplicities. Thus we conclude from (\ref{fc-fc13}) that the polynomial system (\ref{fc-fc11}) has at most $N(\{n_{1,k}\}_{k},\{n_{2,k}\}_{k})$ solutions and so does the Toda system (\ref{Toda-more-general}).
\end{proof}

\subsection{Proofs of Theorem \ref{thm4}-\ref{thm5}}

In this section, we study the special Toda system (\ref{Toda}) and the associated ODE (\ref{fc-3ode}) via Theorem \ref{thm3}, and prove Theorem \ref{thm4}-\ref{thm5}. In this case, since $m=0$, the polynomial system (\ref{fc-fc11}) with $5$ polynomials reduces to a polynomial system with $3$ polynomials of $(B,D_0,D)$ because $B_0=A_0=0$. We summarize this fact as follows.

\begin{lemma}\label{lemma-su3-3} There are three polynomials \[P_{0,1}(B,D_0,D)=r_1D_0^{n_1+1}+\cdots\in\mathbb{Q}[g_2,g_3][B,D_0,D],\;r_1\in\mathbb{Q}\setminus\{0\},\] \[P_{0,2}(B,D_0,D)=r_2D_0^{n_2+1}+\cdots\in\mathbb{Q}[g_2,g_3][B,D_0,D],\;r_2\in\mathbb{Q}\setminus\{0\},\]
\[P_{0,3}(B,D_0,D)=r_3BD_0^{n_1+n_2}+\cdots\in\mathbb{Q}[g_2,g_3][B,D_0,D],\;r_3\in\mathbb{Q}\setminus\{0\},\]
with homogeneous weights
\[
\operatorname{Weig} P_{0,1}=n_1+1,\quad \operatorname{Weig} P_{0,2}=n_2+1,\quad \operatorname{Weig} P_{0,3}=n_1+n_2+2,
\]
such that $0$ is an apparent singularity of ODE (\ref{fc-3ode})
if and only if
\begin{equation}\label{su3-fc42}P_{0,1}(B,D_0,D)=P_{0,2}(B,D_0,D)=P_{0,3}(B,D_0,D)=0.\end{equation}

Consequently, the number of solutions of the Toda system (\ref{Toda})
equals to the number of solutions $(B,D_0,D)$'s of the polynomial system (\ref{su3-fc42}),
and the number of even solutions equals to the number of solutions of the form $(B,0,0)$.
\end{lemma}

\begin{proof} Comparing to Lemma \ref{lemma-su3-33}, we only need to explain the term $BD_0^{n_1+n_2}$ in the expression of $P_{0,3}$ (because $A_0=0$ makes the term $A_0D_0^{n_1+n_2}=0$ in the original expression of $P_{0,3}$ stated in Lemma \ref{lemma-su3-33}).

Recalling the local exponents $\rho_j$ in (\ref{localex0-1}), for ODE (\ref{fc-3ode}) the formula (\ref{eq-appapp}) becomes simpler and reads as
\begin{align}\label{eq-aappapp}&\sum_{j=0}^{\infty}\Big(\phi_jc_j+D_0c_{j-1}-(j+\rho_1-2)Bc_{j-2}+Dc_{j-3}\\
&-\sum_{i=4}^j[(j+\rho_1-i)\alpha-(i-2)\beta]b_ic_{j-i}+D_0\sum_{i=4}^{j-1}b_ic_{j-1-i}\Big)
z^{j+\rho_1-3}=0,\nonumber\end{align}
with
\begin{align*}
\phi_j=j(j-n_1-1)(j-n_1-n_2-2).
\end{align*}
so the recursive formula (\ref{rec-3-app}) becomes
\begin{align}\label{rec-app}
\phi_j c_j=&-D_0c_{j-1}+(j+\rho_1-2)Bc_{j-2}-Dc_{j-3}\\&
+\sum_{i=4}^j[(j+\rho_1-i)\alpha-(i-2)\beta]b_ic_{j-i}-D_0\sum_{i=4}^{j-1}b_ic_{j-1-i},\;\forall j\geq 0.\nonumber
\end{align}
Now we have $\rho_1\notin\mathbb{Z}$, so $j+\rho_1-2\neq 0$ for all $j\geq 0$. Then the induction argument as in Lemma \ref{lemma-su3-33} implies $P_{0,3}(B,D_0,D)=r_3BD_0^{n_1+n_2}+\cdots$ with $r_3\neq 0$.
\end{proof}

\begin{remark}\label{rmkk}
Conjecture \ref{conj-1} is equivalent to assert that the number of solutions of the polynomial system (\ref{su3-fc42}) is exactly
\[N(n_1,n_2)=\frac{(n_1+1)(n_2+1)(n_1+n_2+2)}{6}\]
except for finitely many $\tau$'s modulo $SL(2,\mathbb{Z})$. Counting the number of solutions of polynomial equations is not easy in general.
Here we have some further discussions about Conjecture \ref{conj-1}.

Recalling Lemma \ref{lemma-su3-3} that $P_{0,j}(B,D_0,D)\in\mathbb{Q}[g_2,g_3][B,D_0,D]$, we define
\[\hat{P}_{0,j}(B,D_0,D):=P_{0,j}(B,D_0,D)\Big|_{g_2=g_3=0},\quad j=1,2,3,\]
and consider the corresponding system
\begin{equation}\label{su3-fcc42}\hat{P}_{0,1}(B,D_0,D)
=\hat{P}_{0,2}(B,D_0,D)=\hat{P}_{0,3}(B,D_0,D)=0.\end{equation}
Since letting $g_2=g_3=0$ is equivalent to replacing $\wp(z)$ with $\frac{1}{z^2}$ in ODE (\ref{fc-3ode}), we consider the corresponding ODE
\begin{equation}\label{fc-3o1de}
y'''-(\tfrac{\alpha}{z^2}+B)y'+(\tfrac{-2\beta}{z^3}+\tfrac{D_0}{z^2}+D)y=0.
\end{equation}
Then $0$ is an apparent singularity of (\ref{fc-3o1de}) if and only if (\ref{su3-fcc42}) holds.
\begin{conjecture}\label{conj-2} $0$ is an apparent singularity of (\ref{fc-3o1de}) if and only if $B=D_0=D=0$, or equivalently
 the polynomial system (\ref{su3-fcc42}) has only trivial solution $B=D_0=D=0$.
 \end{conjecture}
Conjecture \ref{conj-2} can be proved for small values of $(n_1,n_2)$ via direct computations. It remains open for general $(n_1,n_2)$ satisfying $n_1\not\equiv n_2\mod 3$.

Now we turn back to the proof of Theorem \ref{thm6} in Section 4.1:
Let $B=\tilde{B}^2$, $D=\tilde{D}^3$ and define
\[\tilde{P}_{0,j}(\tilde{B},D_0,\tilde{D}):=P_{0,j}(B,D_0,D),\quad j=1,2,3.\]
Suppose Conjecture \ref{conj-2} holds. Then the associated homogenized system of
\begin{equation}\label{su3-fc42-1}\tilde{P}_{0,1}(\tilde{B},D_0,\tilde{D})
=\tilde{P}_{0,2}(\tilde{B},D_0,\tilde{D})=\tilde{P}_{0,3}(\tilde{B},D_0,\tilde{D})=0\end{equation}
has \emph{no} solutions at infinity in $\mathbb{CP}^3$, so it follows from the B\'{e}zout theorem (see e.g. \cite[p.246]{IRS}) that the polynomial system (\ref{su3-fc42-1}) has \emph{exactly}
$\prod_{j=1}^3\deg \tilde{P}_{0,j}=6N(n_1,n_2)$ solutions by \emph{counting multiplicities}.

Therefore, Conjecture \ref{conj-1} is equivalent to assert that Conjecture \ref{conj-2} holds and \emph{the multiplicity of any solution of (\ref{su3-fc42-1}) is $1$ except for finitely many $\tau$'s modulo $SL(2,\mathbb{Z})$}, which we strongly believe to be true.

For small values of $(n_1, n_2)$, Conjecture \ref{conj-1} can be proved by direct computations as done below for $n_1=0$ and $n_2=1,2,4$.
\end{remark}

\begin{proof}[Proof of Theorem \ref{thm4}]
(1) Let $(n_1,n_2)=(0,1)$. Then $\rho_1=-\frac{1}{3}$ and a direct computation via the recursive formula (\ref{rec-app}) gives
\[P_{0,1}=-D_0,\quad P_{0,2}=-\frac{D_0^2}{2}+\frac{2B}{3},\quad P_{0,3}=-D-\frac{BD_0}{6}.\]
So (\ref{su3-fc42}) has a unique solution $(0,0,0)$, namely the Toda system (\ref{Toda}) has a unique solution which is even.

(2). Let $(n_1,n_2)=(0,2)$. Then $\rho_1=-\frac{2}{3}$ and so
\[P_{0,1}=-D_0,\quad P_{0,2}=-\frac{D_0^3}{24}+\frac{7BD_0}{18}-D,\]
\[P_{0,3}=-\frac{BD_0^2}{36}-\frac{D_0D}{6}+\frac{2B^2}{9}-\frac{2g_2}{27}.\]
So (\ref{su3-fc42}) has solutions $D_0=D=0$ and $B=\pm\sqrt{g_2/3}$.

Consequently, when $g_2=0$ (i.e. $\tau\equiv e^{\pi i/3}$ mod $SL(2,\mathbb{Z})$), the Toda system (\ref{Toda}) has a unique solution which is even; when $g_2\neq 0$ (i.e. $\tau\not\equiv e^{\pi i/3}$ mod $SL(2,\mathbb{Z})$), the Toda system (\ref{Toda}) has exactly $2$ solutions that are both even.
\end{proof}

\begin{proof}[Proof of Theorem \ref{thm5}]
Let $(n_1,n_2)=(0,4)$. Then $\rho_1=-\frac{4}{3}$. Again $P_{0,1}=-D_0$ while the expressions of $P_{0,2}, P_{0,3}$ are slightly complicated. Instead we may insert $D_0=0$ in the recursive formula (\ref{rec-app}) and obtain
\[P_{0,2}(B,0,D)=\frac{5}{54}BD,\;\, P_{0,3}(B,0,D)=\frac{-1}{486}(6B^3+27D^2-56g_2B+288g_3).\]
It follows from $P_{0,2}(B,0,D)=0$ that $B=0$ or $D=0$.

If $D=0$, we have $6B^3-56g_2B+288g_3=0$, the discriminant of which is $\Delta(\tau):=343g_2^3-6561g_3^2$ up to a nonzero constant. Since this $\Delta(\tau)$ is a modular form of weight $12$ with respect to $SL(2,\mathbb{Z})$, it is standard by the theory of modular forms that it has a unique zero $\tau_0$ modulo $SL(2,\mathbb{Z})$ action. Therefore, $6B^3-56g_2B+288g_3=0$ has $3$ (resp. $2$) distinct roots for $\tau\not\equiv \tau_0$ modulo $SL(2,\mathbb{Z})$ (resp. $\tau\equiv \tau_0$ modulo $SL(2,\mathbb{Z})$), namely the Toda system has exactly $3$ (resp. $2$) even solutions for $\tau\not\equiv \tau_0$ modulo $SL(2,\mathbb{Z})$ (resp. $\tau\equiv \tau_0$ modulo $SL(2,\mathbb{Z})$).

If $D\neq 0$ then $B=0$ and $27D^2+288g_3=0$, so the Toda system has exactly $2$ solutions which are not even as long as $g_3\neq 0$. Noting that $g_3(\tau_0)\neq 0$, we obtain the desired statements of Theorem \ref{thm5}.
\end{proof}

\section{Proof of Theorem \ref{thm1}}

This section is devoted to the proof of Theorem \ref{thm1}.
Consider the special case $m=3$ and $p_k=\frac{\omega_k}{2}$, i.e. the Toda system (\ref{Toda-general}) or equivalently
\begin{equation}\label{TodaUV-e}
\begin{cases}
\Delta U+e^{2U-V}=4\pi \sum_{k=0}^3\gamma_{1,k}\delta_{\frac{\omega_k}{2}}\quad\text{ on
}\; E_{\tau},\\
\Delta V+e^{2V-U}=4\pi \sum_{k=0}^3\gamma_{2,k}\delta_{\frac{\omega_k}{2}}\quad\text{ on
}\; E_{\tau}.
\end{cases}
\end{equation}
If the solution $(U,V)$ of the Toda system (\ref{TodaUV-e}) is even, then $W_2$ is even and $W_3$ is odd, so the associated ODE (\ref{su3-ode}) takes a simpler form (Note $A_0=-A_1-A_2-A_3$)
\begin{align}\label{su3-ode-even}
\mathcal{L}y=y'''&-\bigg(\sum_{k=0}^3\alpha_k\wp(z-\tfrac{\omega_k}{2})+B\bigg)y'\\
&+\bigg(\sum_{k=0}^3\beta_k\wp'(z-\tfrac{\omega_k}{2})+\sum_{k=1}^3A_k
\frac{\wp'(z)}{2(\wp(z)-e_k)}\bigg)y=0,\nonumber
\end{align}
where $e_k=\wp(\frac{\omega_k}{2})$ and we used the formula
\[\zeta(z-\tfrac{\omega_k}{2})-\zeta(z)+\zeta(\tfrac{\omega_k}{2})=\frac{\wp'(z)}{2(\wp(z)-e_k)}.\]

\begin{proof}[Proof of Theorem \ref{thm1}]
By changing variable $z\to z+\frac{\omega_k}{2}$ if necessary, we may assume that $n_{1,0}$ and $n_{2,0}$ are both odd.

Assume by contradiction that the Toda system (\ref{TodaUV-e}) has an \emph{even} solution $(U, V)$. Then it follows from Lemma \ref{thm-A} that $\frac{\omega_k}{2}$'s are apparent singularities of the associated ODE (\ref{su3-ode-even}) for all $k$, and there is a basis of local solutions $\vec{Y}=(y_1, y_2, y_3)^T$ in $B(q_0, |q_0|/2)$ such that
the monodromy matrices $N_1, N_2$ with respect to $\vec{Y}$ are given by (\ref{su3-fc4}).

{\bf Step 1.} Recalling the local exponents $\rho_{0,j}$ in (\ref{localex0}), since $0$ is an apparent singularity and ODE (\ref{su3-ode-even}) is invariant under $z\leftrightarrow -z$, the standard Frobenius' method implies that there is a basis of local solutions $(\eta_1, \eta_2, \eta_3)^T$ of the form
\[\eta_j(z)=z^{\rho_{0,j}}\sum_{l=0}^{+\infty}c_{j,l}z^{2l},\quad c_{j,0}\neq 0,\quad j=1,2,3\]
in a small neighborhood $B(0,\delta_0)=\{z| |z|<\delta_0\}$ of $0$, where we can take $\delta_0\in (0,\frac{1}{2(1+|\operatorname{Im}\tau|)})$. Since $n_{1,0}, n_{2,0}$ are both odd, we see from (\ref{localex0}) that $\rho_{0,j}-\rho_{0,1}$ are positive even integers for $j=2,3$, so $\eta_j(z)$ are all of the form
\begin{equation}\label{su3-fc5}\eta_j(z)
=z^{\rho_{0,1}}\sum_{l=0}^{+\infty}\tilde{c}_{j,l}z^{2l},\quad z\in B(0,\delta_0),\quad j=1,2,3.\end{equation}

On the other hand, we may take $\varepsilon_0<\frac{\delta_0}{4|1+\tau|}$ such that $B(q_0, |q_0|/2)\subset B(0,\delta_0)$. Then there is a matrix $P\in GL(3,\mathbb{C})$ such that
\begin{equation}\label{su3-fc6}
\vec{Y}=(y_1,y_2,y_3)^T=P(\eta_1,\eta_2,\eta_3)^T\quad\text{in }B(q_0,|q_0|/2).
\end{equation}

{\bf Step 2.} We consider simply-connected domains that have no intersections with the set $E_{\tau}[2]$ of singularities of ODE (\ref{su3-ode-even}):
\[\Omega:=\{a+b\tau |a\in \mathbb{R}, b\in (-\tfrac12,\tfrac12)\}\setminus((-\infty,-\tfrac12]\cup [0,+\infty)),\]
\[\Omega_+:=\{a+b\tau |a\in \mathbb{R}, b\in (0,\tfrac12)\},\]
\[\Omega_-:=\{a+b\tau |a\in \mathbb{R}, b\in (-\tfrac12,0)\}=-\Omega_+.\]
Since $B(q_0,|q_0|/2)\subset \Omega_{-}\subset\Omega$, by analytic continuation the local solutions $\vec{Y}=(y_1,y_2,y_3)^T$ in $B(q_0,|q_0|/2)$ can be extended to be \emph{single-valued holomorphic} functions (still denoted by $\vec{Y}=(y_1,y_2,y_3)^T$) in $\Omega$.
Then by (\ref{su3-fc5})-(\ref{su3-fc6}) we have
\begin{equation}\label{su3-fc7}
\vec{Y}=P(\eta_1,\eta_2,\eta_3)^T\quad\text{in }B(0,\delta_0)\cap\Omega=B(0,\delta_0)\setminus [0,+\infty).
\end{equation}

{\bf Step 3.} Since $\Omega_\pm +1=\Omega_\pm$,
$y_j(z+1)$ is well-defined and also a solution of ODE (\ref{su3-ode-even}) for $z\in\Omega_{\pm}$. So there are invertible matrices $N_{\pm}$ such that
\[\vec{Y}(z+1)=N_{\pm}\vec{Y}(z),\quad \forall\, z\in \Omega_{\pm}.\]
Recalling the definition of $\Omega$ and (\ref{su3-fc2}), it is easy to see
\begin{equation}\label{su3-fc31}N_+^{-1}N_-=M_0M_1=e^{-2\pi i (\gamma_{1,0}+\gamma_{1,1})}I_3.\end{equation}
On the other hand, we see from $B(q_0,|q_0|/2)\subset \Omega_{-}$, the definition of the fundamental cycle $\ell_1$ and $N_1=\rho(\ell_1)$ that $N_-=N_1$, so we conclude
\begin{equation}\label{su3-fc8}
\vec{Y}(z+1)=N_1\vec{Y}(z),\quad \forall\, z\in \Omega_{-},
\end{equation}
\begin{equation}\label{su3-fc9}
\vec{Y}(z+1)=e^{2\pi i (\gamma_{1,0}+\gamma_{1,1})}N_{1}\vec{Y}(z),\quad \forall\, z\in \Omega_{+}.
\end{equation}

{\bf Step 4.} Since $\Omega_+=-\Omega_-$, $y_j(-z)$ and $y_j(z)$ are also solutions of  ODE (\ref{su3-ode-even}) in $\Omega_{-}$, so there is a invertible matrix $\tilde{N}$ such that
\[\vec{Y}(-z)=\tilde{N}\vec{Y}(z)\quad \forall\, z\in \Omega_{-}.\]
On the other hand, since (\ref{su3-fc5}) and (\ref{su3-fc7}) together imply
\[\vec{Y}(-z)=\vec{Y}(e^{-\pi i}z)=e^{-\pi i \rho_{0,1}}\vec{Y}(z)\quad\text{for }z\in B(q_0,|q_0|/2),\]
we have $\tilde{N}=e^{-\pi i \rho_{0,1}} I_3$, namely
\begin{equation}\label{su3-fc10}\vec{Y}(-z)=e^{-\pi i \rho_{0,1}}\vec{Y}(z)\quad \forall\, z\in \Omega_{-}.\end{equation}

{\bf Step 5.} By (\ref{su3-fc8})-(\ref{su3-fc9}) and (\ref{su3-fc10}), we have for $z\in\Omega_-$ that
\begin{align*}
\vec{Y}(-(z+1))=e^{-\pi i \rho_{0,1}}\vec{Y}(z+1)=e^{-\pi i \rho_{0,1}}N_1\vec{Y}(z),
\end{align*}
and
\begin{align*}
\vec{Y}(-(z+1))&=\vec{Y}(-z-1)\\
&=e^{-2\pi i (\gamma_{1,0}+\gamma_{1,1})}N_{1}^{-1}\vec{Y}(-z)\\
&=e^{-2\pi i (\gamma_{1,0}+\gamma_{1,1})}e^{-\pi i \rho_{0,1}}N_{1}^{-1}\vec{Y}(z),
\end{align*}
so we obtain $N_1=e^{-2\pi i (\gamma_{1,0}+\gamma_{1,1})}N_{1}^{-1}$, i.e.
\[e^{-2\pi i (\gamma_{1,0}+\gamma_{1,1})}I_3=N_1^2=\left(\begin{smallmatrix}1 & &\\
&\varepsilon & \\
& & \varepsilon^2\end{smallmatrix}\right)^2=\left(\begin{smallmatrix}1 & &\\
&\varepsilon^2 & \\
& & \varepsilon^4\end{smallmatrix}\right),\]
clearly a contradiction because $\varepsilon=e^{-2\pi i \frac{2\mathcal{N}_1+\mathcal{N}_2}{3}}\neq \pm 1$.

Therefore, (\ref{TodaUV-e}) has no even solutions.
\end{proof}

\section{Proof of Theorem \ref{thm2}}

This section is devoted to the proof of Theorem \ref{thm2}. Here instead of counting the solutions of the form $(B,0,0)$ for the polynomial system (\ref{su3-fc42}), we study directly the condition on $B$ such that $0$ is an apparent singularity of (\ref{3ode-even}). Again we apply the standard Frobenius' method.

Recall that since at least one of $n_1,n_2$ is even, the integer $N_e$ is well defined by (\ref{su3-fc20}).

\begin{lemma}\label{lemma-su3-2} Suppose at least one of $n_1,n_2$ is even. Then there is a monic polynomial $P_{N_e}(B)\in\mathbb{Q}[g_2,g_3][B]$ with degree $N_e$ in $B$ and also homogeneous weight $N_e$ such that $0$ is an apparent singularity of (\ref{3ode-even}) if and only if $P_{N_e}(B)=0$.
Here the weights of $B, g_2, g_3$ are $1, 2, 3$ respectively.

\end{lemma}

\begin{proof} Since ODE (\ref{3ode-even}) is invariant with respect to $z\leftrightarrow-z$, it is more convenient for us to
descend (\ref{3ode-even}) to $\mathbb{P}^1$ under the double cover $\wp: E_{\tau}\to \mathbb{P}^1$. Let $x=\wp(z)$, $\tilde{y}(x)=y(z)$ and denote
\[p(x):=4x^3-g_2x-g_3=\wp'(z)^2.\]
Then
\[\frac{dy}{dz}=\wp'(z)\frac{d\tilde{y}}{dx},\qquad
\frac{d^2y}{dz^2}=p(x)\frac{d^2\tilde{y}}{dx^2}+\frac12p'(x)\frac{d\tilde{y}}{dx},\]
\begin{align*}
\frac{d^3y}{dz^3}=\wp'(z)\left[p(x)\frac{d^3\tilde{y}}{dx^3}
+\frac32p'(x)\frac{d^2\tilde{y}}{dx^2}+12x\frac{d\tilde{y}}{dx}\right],
\end{align*}
so
\begin{equation}\label{ndd-ee1}y'''-(\alpha\wp(z)+B)y'+\beta\wp'(z)y=\wp'(z)\tilde{L}\tilde{y},\end{equation}
where
\[\tilde{L}:=p(x)D^3+3(6x^2-\tfrac{g_2}{2})D^2+[(12-\alpha)x-B]D+\beta,\;\, D:=d/dx.\]
Namely $y(z)$ solves (\ref{3ode-even}) if and only if
$\tilde{L}\tilde{y}(x)=0$.

Recalling (\ref{localex0-1}),
the local exponents of $\tilde{L}\tilde{y}(x)=0$ at $x=\infty$ are $\frac{\rho_j}{2}$'s:
\begin{equation}\label{localex0-2}
\frac{\rho_{1}}{2}=-\frac{\gamma_{1}}{2}< \frac{\rho_{2}}{2}=-\frac{\gamma_{1}}{2}+\frac{n_{1}+1}{2}< \frac{\rho_{3}}{2}=-\frac{\gamma_{1}}{2}+\frac{n_{1}+n_2+2}{2}.
\end{equation}
Clearly $0$ is an apparent singularity of (\ref{3ode-even}) if and only if $\infty$ is an apparent singularity of $\tilde{L}\tilde{y}(x)=0$.
Moreover, the standard Frobenius' method says that $\infty$ is an apparent singularity if and only if $\tilde{L}\tilde{y}(x)=0$ has solutions of the form
\begin{equation}\label{su3-fc21}\tilde{y}_k(x)=x^{-\frac{\rho_k}{2}}\sum_{j=0}^{+\infty}C_j x^{-j},\quad C_0=1\end{equation}
for those $k\in\{1,2\}$ satisfying that $\frac{\rho_i}{2}-\frac{\rho_k}{2}\in\mathbb{N}$ for some $i>k$.

Inserting (\ref{su3-fc21}) into $\tilde{L}\tilde{y}_k(x)=0$ we easily obtain
\[0=\tilde{L}\tilde{y}_k(x)=\sum_{j=0}^{+\infty}\Phi_jx^{-j-\frac{\rho_k}{2}},\]
where
\begin{align*}
\Phi_j:=&\phi_jC_j-B(-j-\tfrac{\rho_k}{2}+1)C_{j-1}\\
&-
g_2(-j-\tfrac{\rho_k}{2}+2)(-j-\tfrac{\rho_k}{2}+\tfrac32)(-j-\tfrac{\rho_k}{2}+1)C_{j-2}\\
&-g_3(-j-\tfrac{\rho_k}{2}+3)(-j-\tfrac{\rho_k}{2}+2)(-j-\tfrac{\rho_k}{2}+1)C_{j-3},
\end{align*}
$C_{-3}=C_{-2}=C_{-1}:=0$ and
\begin{align}\label{su3-fc33}\phi_j:=&4(-j-\tfrac{\rho_k}{2})(-j-\tfrac{\rho_k}{2}-1)(-j-\tfrac{\rho_k}{2}-2)\\
&+18(-j-\tfrac{\rho_k}{2})(-j-\tfrac{\rho_k}{2}-1)+(12-\alpha)(-j-\tfrac{\rho_k}{2})+\beta\nonumber\\
=&-4\prod_{i=1}^3(j+\tfrac{\rho_k}{2}-\tfrac{\rho_i}{2}).\nonumber
\end{align}
Therefore, $\tilde{L}\tilde{y}_k(x)=0$ if and only if
\begin{align}\label{su3-fc22}
\phi_jC_j=&B(-j-\tfrac{\rho_k}{2}+1)C_{j-1}\\
&+
g_2(-j-\tfrac{\rho_k}{2}+2)(-j-\tfrac{\rho_k}{2}+\tfrac32)(-j-\tfrac{\rho_k}{2}+1)C_{j-2}\nonumber\\
&+g_3(-j-\tfrac{\rho_k}{2}+3)(-j-\tfrac{\rho_k}{2}+2)(-j-\tfrac{\rho_k}{2}+1)C_{j-3},\;\forall j\geq 0.\nonumber
\end{align}
Note that $-j-\tfrac{\rho_k}{2}+1\neq 0$ for any $j\geq 0$ because $\tfrac{\rho_k}{2}\notin \mathbb{Z}$.

{\bf Case 1.} $n_1$ is odd and $n_2$ is even.

Then $N_e=\frac{n_1+1}{2}$ and $\frac{\rho_i}{2}-\frac{\rho_k}{2}\in\mathbb{N}$ if and only if $(k,i)=(1,2)$. Let $k=1$ in (\ref{su3-fc33})-(\ref{su3-fc22}), we have
\begin{equation}\label{su3-fc23}
\phi_j=-4j(j-\tfrac{n_1+1}{2})(j-\tfrac{n_1+n_2+2}{2})=-4j(j-N_{e})(j-\tfrac{n_1+n_2+2}{2}),
\end{equation}
so $\phi_j=0$ with $j\geq 0$ if and only if $j\in \{0,N_e\}$.

Clearly (\ref{su3-fc22}) with $j=0$ holds automatically. For $1\leq j\leq N_e-1$, by (\ref{su3-fc22}) and the induction argument, $C_j$ can be uniquely solved as $C_j=C_j(B)\in \mathbb{Q}[g_2,g_3][B]$ with degree $j$ in $B$ and homogeous weight $j$. Consequently, the RHS of (\ref{su3-fc22}) with $j=N_e$ is a polynomial in $\mathbb{Q}[g_2,g_3][B]$ with degree $N_e$ and homogenous weight $N_e$. Define $P_{N_e}(B)$ to be the corresponding monic polynomial. Then the standard Frobenius theory shows that  $\tilde{L}\tilde{y}_1=0$ has a local solution $\tilde{y}_1(z)$ of the form (\ref{su3-fc21}) if and only if  $P_{N_e}(B)=0$. This proves that $\infty$ is an apparent singularity if and only if $P_{N_e}(B)=0$.

{\bf Case 2.} $n_1$ is even and $n_2$ is odd.

Then $N_e=\frac{n_2+1}{2}$ and $\frac{\rho_i}{2}-\frac{\rho_k}{2}\in\mathbb{N}$ if and only if $(k,i)=(2,3)$. Let $k=2$ in (\ref{su3-fc33})-(\ref{su3-fc22}), we have
\[\phi_j=-4j(j+\tfrac{n_1+1}{2})(j-\tfrac{n_2+1}{2})=-4j(j+\tfrac{n_1+1}{2})(j-N_{e}),\]
so $\phi_j=0$ with $j\geq 0$ if and only if $j\in \{0,N_e\}$. The rest proof is the same as Case 1.

{\bf Case 3.} $n_1$ and $n_2$ are both even.

Then $N_e=\frac{n_1+n_2+2}{2}$ and $\frac{\rho_i}{2}-\frac{\rho_k}{2}\in\mathbb{N}$ if and only if $(k,i)=(1,3)$. Let $k=1$ in (\ref{su3-fc33})-(\ref{su3-fc22}), we have
\[\phi_j=-4j(j-\tfrac{n_1+1}{2})(j-\tfrac{n_1+n_2+2}{2})=-4j(j-\tfrac{n_1+1}{2})(j-N_{e}),\]
so $\phi_j=0$ with $j\geq 0$ if and only if $j\in \{0,N_e\}$. Again the rest proof is the same as Case 1.
\end{proof}

\begin{lemma}\label{lemma-real}
If one of the following holds:
\begin{itemize}
\item[(1)] $n_1$ is odd, i.e. $N_e=\frac{n_1+1}{2}$,
\item[(2)] $n_2$ is odd (i.e. $N_e=\frac{n_2+1}{2}$) and $n_2-n_1\in \{1,5\}$,
\item[(3)] $n_1$ is even and $n_2=n_1+2$, i.e. $N_e=\frac{n_1+n_2+2}{2}=n_1+2$,
\end{itemize}
then the polynomial $P_{N_e}(B)$ in Lemma \ref{lemma-su3-2} has $N_e$ distinct roots expect for finitely many $\tau$'s modulo $SL(2,\mathbb{Z})$ action.
\end{lemma}

\begin{proof}

{\bf Step 1.} We consider the special case $\tau=i=\sqrt{-1}$ and prove that $P_{N_e}(B)$ has $N_e$ real distinct roots.
We only prove this assertion for the condition (1). The cases for conditions (2)-(3) can be proved similarly and we leave the details to the interested reader.

For $\tau=i$, it is well known that $g_3=0$ and $g_2>0$.
Then under condition (1), we see from  (\ref{su3-fc22})-(\ref{su3-fc23}) and $\rho_1=-\frac{2n_1+n_2}{3}$ that
\begin{align}\label{su3-fc22-0}
\phi_jC_j=&B(-j+\tfrac{2n_1+n_2}{6}+1)C_{j-1}+
g_2(-j+\tfrac{2n_1+n_2}{6}+2)\\
&\cdot(-j+\tfrac{2n_1+n_2}{6}+\tfrac32)(-j+\tfrac{2n_1+n_2}{6}+1)C_{j-2},\;\forall j\geq 0,\nonumber
\end{align}
\begin{equation}\label{su3-fc24}
\phi_j=-4j(j-N_{e})(j-\tfrac{n_1+n_2+2}{2}).
\end{equation}
Since $N_e=\frac{n_1+1}{2}$ and $n_1<n_2$, we have
\[\phi_j<0\quad \forall 1\leq j\leq N_e-1,\]
\[-j+\tfrac{2n_1+n_2}{6}+1>0\quad \forall 1\leq j\leq N_e.\]
Recall that $C_0=1$, $C_{-3}=C_{-2}=C_{-1}=0$, and $C_j=C_j(B)$ are polynomials in $\mathbb{Q}[g_2][B]\subset \mathbb{R}[B]$ with degree $j$ for $1\leq j\leq N_e-1$. For convenience we also denote by $-C_{N_e}(B)$ to be the RHS of (\ref{su3-fc22-0}) with $j=N_e$, i.e.
\begin{align*}
-C_{N_e}(B):=&B(-N_e+\tfrac{2n_1+n_2}{6}+1)C_{N_e-1}+
g_2(-N_e+\tfrac{2n_1+n_2}{6}+2)\\
&\cdot(-N_e+\tfrac{2n_1+n_2}{6}+\tfrac32)(-N_e+\tfrac{2n_1+n_2}{6}+1)C_{N_e-2}.
\end{align*}
Then $C_{N_e}(B)=c P_{N_e}(B)$ for some $c\neq 0$.

In view of the above argument, it is easy to see that the following properties hold for $1\leq j\leq N_e$:

{\bf (P1)} Up to a positive constant, the leading term of $C_{j}(B)$ is $(-1)^{j}B^{j}$.

{\bf (P2)} If $C_{j-1}(B)=0$ and $C_{j-2} (B)\neq 0$ for $B\in \mathbb{R}$, then
$C_{j} (B)C_{j-2} (B)<0$.

Then by an induction argument, it is easy to prove that for $1\leq j\leq N_e$, $C_j(B)$ has $j$ real distinct roots, denoted by $r_1^j<\cdots<r_{j}^j$, such that
\begin{equation}\label{root1}
r_1^j<r_1^{j-1}<r_2^j<\cdots<r_{j-1}^{j}<r_{j-1}^{j-1}<r_{j}^j.
\end{equation}
See e.g. \cite{CKLT}. We give the details here for completeness.

The case $j=1$ is trivial because $\deg C_{1}(B)=1$. Let $2\leq m\leq N_e-1$ and
assume that the statement is true for any $j\leq m-1$. We prove it for $j=m$. From the assumption of the
induction, we have
\begin{equation}\label{s-k}
r_{1}^{m-1}<r_{1}^{m-2}<r_{2}^{m-1}<\cdots
<r_{m-2}^{m-1}<r_{m-2}^{m-2}<r_{m-1}^{m-1}.
\end{equation}
Recall {\bf (P1)} that
\begin{equation}\label{ck-1}
\lim_{B\rightarrow-\infty} C_{m-2}(B) =+\infty, \; \lim_{B \rightarrow
+\infty} C_{m-2}(B) = (-1)^{m-2}\infty.
\end{equation}
Since $r_{j}^{m-2}$, $1\le j\le m-2$, are all the roots of  $C_{m-2}(B)$, it follows from (\ref{s-k}) and (\ref{ck-1}) that
\begin{equation}\label{c-k-1-k}
C_{m-2}(r_j^{m-1})\sim (-1)^{j-1},\quad \forall j\in [1,m-1].
\end{equation}
Here $c\sim (-1)^j$ means $c=(-1)^j\tilde{c}$ for some $\tilde{c}>0$. Then we see from {\bf (P2)} that
\[C_{m}(r_j^{m-1})\sim (-1)^{j},\quad \forall j\in [1,m-1].\]
On the other hand, {\bf (P1)} implies
\[\lim_{B \rightarrow-\infty} C_{m}(B) =+\infty, \; \lim_{B \rightarrow
+\infty} C_{m}(B) = (-1)^{m} \infty.\]
From here,
it follows from the intermediate value theorem that the polynomial
$C_{m}(B)$ has $m$ real distinct roots $r_{j}^{m} $ $(1\leq j\leq
m)$ such that \begin{equation}\label{sk1}r_{1}^{m}<r_{1}^{m-1}<r_{2}^{m}
<\dots<r_{m-1}^{m}<r_{m-1}^{m-1}<r_{m}^{m}.\end{equation}
This proves (\ref{root1}) for all $1\leq j\leq N_e$.
In particular, $P_{N_e}(B)$ has $N_e$ real distinct roots because $C_{N_e}(B)=c P_{N_e}(B)$ for some $c\neq 0$.

{\bf Step 2.} We complete the proof.

Recall that $g_2=g_2(\tau), g_3=g_3(\tau)$ are modular forms of weights $4, 6$ respectively, with respect to $SL(2,\mathbb{Z})$.
By Lemma \ref{lemma-su3-2} and Step 1, the discriminant of $P_{N_e}(B)$ is a nonzero modular form with respect to $SL(2,\mathbb{Z})$ and so has only finitely many zeros $\tau$'s modulo $SL(2,\mathbb{Z})$ action. This implies that $P_{N_e}(B)$ has $N_e$ distinct roots expect for finitely many $\tau$'s modulo $SL(2,\mathbb{Z})$ action.
\end{proof}

\begin{remark}\label{rmk}
We believe that the assertion of Lemma \ref{lemma-real} should hold without the conditions (1)-(3). For other cases than (1)-(3), the above induction argument via the recursive formula does not apply. For example, let us consider the case that $n_2$ is odd, $n_1$ is even and $n_2-n_1\geq 7$. Then for $\tau=i$,  we will obtain from the recursive relation  (\ref{su3-fc22}) that $C_2(B)=d_1B^2+d_2g_2$ with some $d_1,d_2\in\mathbb{Q}\setminus\{0\}$ satisfying $d_1d_2>0$, namely $C_2(B)$ has no real roots and so the induction argument fails.  Thus different ideas are needed to settle this problem, which remains open.
\end{remark}

\begin{proof}[Proof of Theorem \ref{thm2}]
Clearly by Theorem \ref{thm3} and Lemma \ref{lemma-su3-2}, we see that the number of even solutions of the Toda system (\ref{TodaUV}) equals to the number of distinct roots of the polynomial $P_{N_e}(B)$. Therefore, Theorem \ref{thm2} follows directly from Lemma \ref{lemma-real}.
\end{proof}

\medskip
\noindent{\bf Acknowledgements} The authors wish to thank the referees very
much for valuable comments. In particular, one referee suggests proposing Conjecture \ref{conj-11}. The research of Z. Chen was supported by NSFC (No. 12071240).

\end{document}